\documentclass[final]{siamltex}  
\usepackage{latexsym}
\usepackage{amsfonts}
\usepackage{amssymb}
\usepackage{subeqnarray}
\usepackage{graphicx}
\usepackage{amscd}
\usepackage{amsmath}
\usepackage{amsfonts}
\usepackage{amssymb}
 \usepackage{url}

\newcommand{\RR}{\mathbb{R}}
\newcommand{\p}{\rm {I\kern-1pt P}}

\newcommand{\EQ}{\begin{equation}\begin{array}{lllllllll}}
\newcommand{\EE}{\end{array}\end{equation}}
\newcommand{\MT}{\left[ \begin{array}{ccccccccc}}
\newcommand{\EM}{\end{array}\right]}

\newenvironment{example}{{\it Example. }}{\hfill $\triangle$ \medskip}

\newcommand{\LM}{\left[\begin{array}{ccccccccc}}
\newcommand{\RM}{\end{array}\right]}
\newcommand{\LA}{\left\{ \begin{array}{ccccccccc}}
\newcommand{\RA}{\end{array}\right.}
\newcommand{\RAA}{\end{array}\right\}}

\newenvironment{remarks}{{\it Remarks.}}{\hfill $\triangleleft$ \smallskip}

\newcommand{\pfrac}[2]{\frac{\partial#1}{\partial#2}}

\newcommand{\bbone}{\mathbf{1}}
\newcommand{\bPsi}{\boldsymbol{\Psi}}
\newcommand{\bPhi}{\boldsymbol{\Phi}}
\newcommand{\bc}{\mathbf{c}}
\newcommand{\bC}{\mathbf{C}}
\newcommand{\bk}{\mathbf{k}}
\newcommand{\cH}{\mathcal{H}}
\newcommand{\cF}{\mathcal{F}}

\newcommand{\cX}{\mathcal{X}}

\newcommand{\tr}{\top\!}
\newcommand{\scal}[2]{\left\langle{#1},{#2}\right\rangle}
\newcommand{\nor}[1]{\left\|{#1}\right\|}

\def\[{\begin{equation}}
\def\]{\end{equation}}
\def\NN{\mbox{I}\!\mbox{N}}

\def\[{\begin{equation}}
\def\]{\end{equation}}

\newcommand{\eq}{\begin{equation}\begin{array}{lclllllllllllllll}}
\newcommand{\ee}{\end{array}\end{equation}}
\newcommand{\bmt}{\left[ \begin{array}{ccccccccc}}
\newcommand{\emt}{\end{array}\right]}
\newcommand{\bea}{\begin{eqnarray}}
\newcommand{\eea}{\end{eqnarray}}
\newcommand{\bean}{\begin{eqnarray*}}
\newcommand{\eean}{\end{eqnarray*}}

\date{}
 
\title{\bf Kernel Methods for the \\ Approximation of Nonlinear Systems}

\author{Jake Bouvrie\footnotemark[2]\ \footnotemark[5]
\and Boumediene Hamzi\footnotemark[3]\ \footnotemark[5]
 }

\begin{document}
\maketitle

\renewcommand{\thefootnote}{\fnsymbol{footnote}}

\footnotetext[2]{Laboratory for Computational and Statistical Learning, Massachusetts Institute of Technology, Cambridge, MA 02138, USA. (jvb@csail.mit.edu)
}
\footnotetext[3]{Department of Mathematics, Imperial College London, London SW7 2AZ, UK.  (b.hamzi@imperial.ac.uk)}
\footnotetext[5]{Parts of this work were done while both authors were at Department of Mathematics, Duke University, Durham, NC 27708, USA.}

\renewcommand{\thefootnote}{\arabic{footnote}}



\maketitle
\begin{abstract}
We introduce a data-driven model approximation method for nonlinear control systems, drawing
on recent progress in machine learning and statistical dimensionality reduction.
The method is based on embedding the nonlinear system in a high (or infinite) dimensional reproducing kernel Hilbert space (RKHS) where linear balanced truncation
may be carried out implicitly. This leads to a nonlinear reduction map which can be combined
with a representation of the system belonging to a RKHS to give a closed, reduced order dynamical system
which captures the essential input-output characteristics of the original model.
Working in RKHS
provides a convenient, general functional-analytical framework for theoretical understanding.
Empirical simulations illustrating the approach are also provided.
\end{abstract}

\section{Introduction}

Data-based modelling of nonlinear dynamical systems
has been addressed by many authors. For example,  several methods have
been developed in Time Series Analysis (\cite{kantz} for example) and System Identification (\cite{ljung}, \cite{box}  for
example). Coifman et al. discuss data-based modelling of a stochastic
Langevin system \cite{Mauro08}. Archambeau et al. \cite{archambeau} proposed methods to
approximate SDEs from data. Smale and Zhou  use kernel methods to
approximate a hyperbolic dynamical system \cite{hyperbolic}. 

In this paper we propose a scheme for the approximation of nonlinear systems using balanced model-order reduction. A key, and to our knowledge, novel point of departure from the literature
on nonlinear model reduction is that our approach marries approximation and dimensionality reduction methods known to the machine learning and statistics communities with existing ideas in linear and nonlinear control. In particular, we apply a method similar to kernel PCA (Principal Component Analysis) as well as function approximation in Reproducing Kernel Hilbert Spaces (RKHSes) to the problem of balanced model reduction. Working in RKHSes
provides a convenient, general functional-analytical framework for theoretical understanding
as well as a ready source of existing results and error estimates. The approach presented here is also strongly empirical, in that observability and controllability, and in some cases the dynamics of the nonlinear system are estimated from simulated or measured trajectories. 

The approach we propose begins by  viewing the controllability and observability energies for nonlinear systems as Gramians in a high (or possibly infinite) dimensional RKHS. These Gramians are approximated empirically and then simultaneously
diagonalized in order to identify directions which, in the RKHS, are both the most
observable and the most controllable. 

The assumption that it is possible to apply the method of linear balancing to a nonlinear system when lifted to a RKHS is far more reasonable than applying the linear theory in the original space hoping for the best. Working in the high dimensional RKHS allows one to
perform linear operations on a representation of the system's state and output which can capture strong nonlinearities. Moreover, working in an RKHS readily include linear spaces\footnote{When using the embedding $\Phi: x \mapsto x$ and kernel $k(x,y)=\langle x,y \rangle$, cf. Appendix A for meaning of these objects.} and, therefore, our approach not only covers existing Linear Theory but also extends the range of available spaces where it is reasonable to assume that one will obtain better results when dealing with a nonlinear problem. Therefore, a system for which existing model reduction methods fail, may be approximated by a lower dimensional system when mapped into a RKHS. This situation closely parallels the problem of linear separability in data classification: A dataset which is not linearly separable might be easily
separated when mapped into a nonlinear feature space. The decision boundary is linear in this feature space, but is nonlinear in the original data space. Essentially, we are proposing to apply linear methods to nonlinear systems once mapped into a high (possibly infinite-dimensional) Hilbert space\footnote{Let us note here that it is important to choose the right RKHS in order to perform this sort of computations. In our work, we used existing Universal Kernels such as the Gaussian or polynomial kernels (cf. Proposition \ref{prop1}) but, in general, properly choosing the right RKHS is an open problem even in classical Learning Theory. }. 

Nonlinear reduction of the state space already allows to the design of simpler controllers,
but is only half of the picture. One would also like to be able to write a closed, reduced dynamical
system whose input-output behavior closely captures that of the original system. This problem is the
focus of the second half of our paper, where we again exploit helpful properties of RKHS in order to provide such a closed system.

The paper is organized as follows. In the next section we provide the relevant background for
model reduction and balancing for linear and nonlinear control systems. We then adapt and extend balancing techniques described in the background section to the current RKHS setting in Section~\ref{sec:empirical_gramians}. Section~\ref{sec:closed_sys} then
proposes a method for determining a closed, reduced nonlinear control system in light of the reduction
map described in~Section~\ref{sec:empirical_gramians}.
Finally,
 Section~\ref{sec:expts} provides experiments illustrating an application of the proposed methods to
some nonlinear systems where the method of linear balancing does not apply in $\RR^n$ since the systems we simulated are not linearly controllable and the origin is not asymptotically stable but the same method of linear balancing applies to the nonlinear system after being lifter to a RKHS. Appendix \ref{sec:defins} contains a review of Learning Theory, Appendix B contains a description of kernel PCA.

 Preliminary results of this work can be found in \cite{allerton, acc2012}.

\section{Background}
Several approaches have been proposed for the reduction of
linear control systems in view of control, but few
exist for finite or infinite-dimensional 
nonlinear control systems. For linear systems, the pioneering ``Input-
Output balancing'' approach proposed by B.C. Moore \cite{moore} observes that the important states are the
ones that are both easy to reach and that generate a lot of energy at the output. If a large amount of energy is required to reach a certain state but the same state yields a small output energy, the state is unimportant for the input-output behavior of the system.
The goal is then to find the states that are \emph{both} the most controllable and the most observable. One way to determine such states is to find a change of coordinates where the controllability and observability Gramians (which can be viewed as a measure of the controllability and the observability of the system) are equal and diagonal. States that are
difficult to reach and that don't significantly affect the output are then ignored or truncated. A system expressed in the coordinates where each state is equally controllable and observable is called its \emph{balanced realization}.

A proposal for generalizing this approach to nonlinear control systems was advanced by J. Scherpen~\cite{scherpen_thesis}, where suitably defined controllability and observability energy functions reduce to Gramians in the linear case. In general, to find the balanced realization of a system one needs to solve a set of Hamilton-Jacobi and Lyapunov equations (as we will discuss below). Moore~\cite{moore} proposed an alternative, data-based approach for balancing in the linear case. This method uses samples of the impulse response of a linear system to construct empirical controllability and observability Gramians which are then balanced and truncated using Principal Components Analysis (PCA, or Proper Orthogonal Decomposition (POD) \cite{jolliffe}). This data-driven strategy was then extended to nonlinear control systems with a stable linear approximation by Lall et al.~\cite{lall}, by effectively applying Moore's method to a nonlinear system by way of the Galerkin projection. Despite the fact that the balancing theory underpinning their approach assumes a linear system, Lall and colleagues were able to effectively reduce some nonlinear systems.

Phillips et al.~\cite{Phillips}  has also studied reduction of nonlinear circuit models
in the case of linear but unbalanced coordinate transformations and found
that approximation using a polynomial RKHS
could offer computational advantages. Gray and Verriest mention
in~\cite{gray} that studying algebraically defined Gramian operators
in RKHS may provide advantageous approximation properties, though the idea
is not further explored. Finally, Coifman et al.~\cite{Mauro08} discuss
 reduction of an uncontrolled stochastic Langevin system. There, eigenfunctions of a
combinatorial Laplacian, built from samples of trajectories, provide a
set of reduction coordinates but does not provide a reduced system. This method is
related to kernel principal components (KPCA) using a Gaussian kernel, however
reduction in this study is carried out on a simplified linear
system outside the context of control.

In the following sections we review  balancing of linear and nonlinear systems as introduced in~\cite{moore} and~\cite{scherpen_thesis}. See also \cite{scherpen_survey} for a good survey on balancing for linear and nonlinear systems.

\subsection{Balancing of Linear Systems}
Consider a linear control system \[\label{linsys}\begin{array}{rcl}\dot{x}&=&Ax+Bu,\\y&=&Cx, \end{array}, \]
where $(A,B)$ is controllable, $(A,C)$ is observable and $A$ is Hurwitz. We define the controllability and the observability Gramians as, respectively,
\[ \label{gramians}\begin{array}{rcl}
W_c=\int_0^{\infty}e^{At}BB^{\tr}e^{A^{\tr}t}\, dt,\quad \mbox{and} \quad W_o=\int_0^{\infty}e^{A^{\tr}t}C^{\tr}Ce^{At}\, dt.
\end{array}\nonumber\]
These two matrices can be viewed as a measure of the controllability and the observability of the system~\cite{moore}. For instance, consider the past energy~\cite{scherpen_thesis, scherpen_balancing}, $L_c(x_0)$, defined
as the minimal energy required to reach $x_0$ from $0$ in infinite time
\[\label{L_c}
L_c(x_0)=\inf_{\substack{
u \in { L}_2(-\infty,0),\\ x(-\infty)=0, x(0)=x_0}}
\frac{1}{2}\int_{-\infty}^0||u(t)||^2\,dt,
\]
 and  the future energy~\cite{scherpen_thesis, scherpen_balancing}, $L_o(x_0)$, defined as the output energy generated
by releasing the system from its initial state $x(t_0)=x_0$, and zero input $u(t)=0$ for $t\ge0$, i.e.
 \[\label{L_o}
 L_o(x_0)=\frac{1}{2}\int_{0}^{\infty}||y(t)||^2\,dt,
 \]
for $x(t_0)=x_0$ and $u(t)=0, t\ge0$.
In the linear case, it can be shown that \[\label{lin_energies} L_c(x_0)=\tfrac{1}{2}x_0^{\tr}W_c^{-1}x_0, \quad \mbox{and} \quad L_o(x_0)=\tfrac{1}{2}x_0^{\tr}W_o x_0.\] The columns of $W_c$ span the controllable subspace while the nullspace of $W_o$ coincides with the unobservable subspace. As such, $W_c$ and $W_o$ (or their estimates) are the key ingredients in many model reduction techniques.
It is also well known that $W_c$ and $W_o$ satisfy the Lyapunov equations~\cite{moore}
\[\label{gramians_lyap}
\begin{array}{rcl}
AW_c+W_cA^{\tr}=-BB^{\tr},\quad A^{\tr}W_o+W_oA=-C^{\tr}C.
\end{array}
\]
Several methods have been developed to solve these equations (see ~\cite{hammarling,laub,li} for example).
As mentioned at the beginning of this section, it is also possible to estimate the Gramians from empirical data, cf. (\ref{wc_empirical}) and (\ref{wo_empirical}) below.

The idea behind balancing is to find a representation where the system's observable and controllable subspaces are aligned so that reduction, if possible, consists of eliminating the states that are least controllable and which are also the least observable. More formally, we would like to find a new coordinate system such that $\Sigma:=W_c=W_o=\mbox{diag}\{\sigma_1,\cdots,\sigma_n\}$
where $\sigma_1 \ge \sigma_2 \ge \cdots \ge \sigma_n > 0$.

\begin{theorem}\cite{dullerud}
If $(A,B)$ is controllable and $(A,C)$ is observable, then the eigenvalues of $W_oW_c$ are similarity invariants, i.e. they are independent of the choice of the state-space representation of (\ref{linsys}).
Moreover, there exists a state-space representation where
\[\Sigma:=W_c=W_o=\mbox{diag}\{\sigma_1,\cdots,\sigma_n\}\]
with $\sigma_1\ge \sigma_2\ge\cdots \ge \sigma_n > 0 $ are the square roots of the eigenvalues of $W_oW_c$. Such representations are called
balanced, and the system is in balanced form. The $\sigma_i$'s, $i = 1, ... ,n$ are called the Hankel singular
values.  
The state space expressed in the transformed coordinates $(QAQ^{-1},QB,CQ^{-1})$ is
balanced and $QW_cQ^{\tr}=Q^{-{\tr}}W_oQ^{-1}=\Sigma$ where $Q \in \RR^{n \times n}$. 
\end{theorem}

The Hankel singular values,  $\sigma_i|_{i = 1}^n$, are the square roots of the eigenvalues of $W_oW_c$ and are the singular values of the Hankel operator \[\label{hankel_lin} \mathbb H =\Psi_o\circ  \Psi_c\] that characterizes the input-output behaviour of the system (\ref{linsys}) with $\Psi_c$, the controllability operator, which maps $u \in L_2(-\infty,0]$ to $x(0)$ and $\Psi_o$, the observability operator, which maps $x(0)$ to $y(t)$, $t\ge0$ with no input applied for $t \ge 0$ \cite{dullerud}. More precisely,
\[\label{psi_c} \begin{array}{rcl}\Psi_c:  L_2(-\infty,0]  &\rightarrow&  \mathbb{C}^n\\ u &\mapsto & \int_{-\infty}^0 e^{-A \tau}Bu(\tau) d\tau  \end{array}  \]
and
 \[\label{psi_o} \begin{array}{rcl}\Psi_o: \mathbb{C}^n &\rightarrow& L_2[0,\infty)\\ x(0)=x_0 &\mapsto & \left\{\begin{array}{l}Ce^{At}x_0, \;\mbox{for} \; t \ge 0,\\ 0, \; \mbox{otherwise}. \end{array} \right.  \end{array}  \]
Clearly, $x_0=\Psi_c u(t)$ for $u(t) \in L_2(-\infty,0]$ is the system state at $t=0$ due to the past input and $y(t)=\Psi_0x_0$, $t \ge 0$ is the future output due to the initial state $x_0$ with the input set to zero.
In fact, $\mathbb H$ characterizes the system's future output $y(t)=\mathbb H u(t), \; t\ge0$ based on the past input $u(t), t \le 0$. More precisely, if $x(-\infty)=0$, \[ \mathbb H u(t)=\Psi_o\Psi_c u(t)= \int_{-\infty}^0Ce^{A(t-\tau)}Bu(\tau)d\tau, \quad \mbox{for} \; t\ge 0\]


When $\mathbb H$ is known to be a compact operator, then its adjoint operator $\mathbb H^{\ast}$ is also compact and the composition $\mathbb H^{\ast} \mathbb  H$ is a self-adjoint compact operator with the spectral decomposition
\[\label{hankel_lin1} {\mathbb H^{\ast} \mathbb H}=\sum_{i=1}^{\infty} \sigma_i^2 \langle \cdot, \Psi_i \rangle_{L_2} \Psi_i, \quad \sigma_i \ge 0, \] 
\[\label{hankel_lin2} \langle \Psi_i, \Psi_j\rangle_{L_2}=\delta_{ij}, \quad \langle \Psi_i, ({\mathbb H^{\ast} \mathbb H})(\Psi_i) \rangle_{L_2}=\sigma_i^2 \]
where $\sigma_i^2$ is an eigenvalue of ${\mathbb H^{\ast} \mathbb H}$ with the corresponding eigenvector $\Psi_i$, ordered as $\sigma_1 \ge \cdots \ge \sigma_n>0$ and $\sigma_{i \ge n+1}=0$ are the Hankel singular values of the input-output system $\Sigma$. For square linear systems,  the nonzero eigenvalues of the Hankel operator associated to the system are the  nonzero eigenvalues 
of the cross Gramian defined as the solution, $W_x$, of $A W_x+W_x A+BC=0 $   \cite{scherpen_survey}.

We also have, for every $x_0 \in {\mathbb C}^n$, \[\label{psic_psio_lin} \Psi_c\Psi_c^{\ast}x_0=W_cx_0, \quad \Psi_o^{\ast}\Psi_ox_0=W_ox_0.\]
Thus $W_c$ and $W_o$ are the matrix representations of the operators $\Psi_c\Psi_c^{\ast}$ and $\Psi_o^{\ast}\Psi_o$ \cite{dullerud}.

For model reduction, typically one looks for a gap in the singular values $\{\sigma_i\}$ for guidance as to where truncation should occur. If  there is a  $k$ such that $\sigma_k \gg\sigma_{k+1}$, then the states most responsible for governing the input-output relationship of the system are $(x_1,\cdots,x_k)$ while $(x_{k+1},\ldots,x_n)$ are assumed to make negligible contributions.

\begin{theorem}\label{thm:error_hankel}\cite{dullerud} Consider a linear system (\ref{linsys}) with its associated  Hankel operator ${\mathbb H}$ (\ref{hankel_lin}). Let ${\mathbb H}_k$ be the Hankel operator of the  reduced order linear system of order $k$. Then, \[||{\mathbb H}-{\mathbb H}_k||=\sigma_{k+1} \]
If $x(0)=0$, the error between  $y$, the output of the full order system, and $y_r$, the output of the reduced order system with $k$ state variables, satisfies \[\label{error_y_yr} ||y-y_r||_2 \le 2 \bigg( \sum_{j=k+1}^n \sigma_j\bigg) ||u||_2 \]

\end{theorem}

If $F$ is unstable then the controllability and observability quantities defined in~\eqref{gramians} are undefined since the integrals will be unbounded. There may, however, still exist solutions to the Lyapunov equations~\eqref{gramians_lyap} when $F$ is unstable ~\cite{therapos, kenney}. Other approaches to balancing unstable linear systems exist (see~\cite{verriest1, verriest2, jonckheere,weiland} for the method of LQG balancing for example). 

Although several methods also exist for computing $Q$~\cite{laub,li}, it is common to simply compute the Cholesky decomposition of $W_o$ so that $W_o=Z Z^{\tr}$, and form the SVD $U \Sigma^2 U^{\tr}$ of  $Z^{\tr} W_c Z$. Then $Q$ is given by $Q=\Sigma^{\frac{1}{2}}U^{\tr}Z^{-1}$.
We also note that the problem of finding the coordinate change $Q$ can be seen as an optimization problem~\cite{antoulas} of the form
$\min_{Q} \mbox{trace}[QW_cQ^{\ast}+Q^{-\ast}W_oQ^{-1}]$. 

\subsection{Balancing of Nonlinear Systems}
In the nonlinear case, the energy functions $L_c$ and $L_o$ in (\ref{L_c}) and (\ref{L_o}) are obtained by solving both a Lyapunov and a Hamilton-Jacobi equation. Here we follow the development of Scherpen~\cite{scherpen_thesis, scherpen_survey}. Consider the nonlinear system
\[\Sigma: \label{sigma}
\left\{\begin{array}{rcl}\dot{x}&=&f(x)+\sum_{i=1}^mg_i(x)u_i,\\ y &=& h(x), \end{array}\right.
\]
with $x \in \RR^n$, $u \in \RR^m$, $y\in \RR^p$, $f(0)=0$, $g_i(0)=0$ for $1 \le i \le m$, and $h(0)=0$.  Moreover, assume the following Hypothesis.\\
{\it Assumption A:} The linearization of~\eqref{sigma} around the origin is controllable, observable and $A=\frac{\partial f}{\partial x}|_{x=0}$ is asymptotically stable.

The controllability operator $\Psi_c: {\cal U} \rightarrow X$ with $X=\RR^n$ and ${\cal U}=L_2^m[0,\infty)$, and the observability operator $\Psi_o: X \rightarrow Y$ with $Y=L_2^p[0,\infty)$ for this system are defined by
\[\Psi_c: u \mapsto x^0: \left\{\begin{array}{l}\dot{x}=-f(x)-g(x)u, \quad x(\infty)=0,\\ x^0=x(0) \end{array} \right. \]
\[\Psi_o:x^0 \mapsto y: \left\{\begin{array}{l}\dot{x}=f(x), \quad x^0=x(0),\\  y=h(x)\end{array} \right. \]
As in the linear case, $\Psi_c$ and $\Psi_o$ represent the input-to-state behavior and the state-to-output behavior, respectively. The Hankel operator for the nonlinear system $\Sigma$ in (\ref{sigma}) is given by the composition of $\Psi_c$ and $\Psi_o$
\[\label{hankel_nonlinear}{\mathbb H} := \Psi_o \circ \Psi_c \]
Consider the norm-minimizing inverse $\Psi_c^{\dag}: X \rightarrow {\cal U}$ 
\[\Psi_c^{\dag}: x^0 \mapsto u:=\mbox{argmin}_{\Psi_c(u)=x^0}||u||. \]
From this point of view $L_c$ in (\ref{L_c}) and $L_o$ in (\ref{L_o}) are
\[L_c(x^0):=\frac{1}{2} ||\Psi_c^{\dag}(x^0) ||^2, \quad L_o(x^0):=\frac{1}{2}||\Psi_o(x^0)||^2\]

\begin{theorem}\label{thm:scherp1}\cite{scherpen_thesis,  scherpen_balancing}
 Consider the nonlinear system $\Sigma$ defined in (\ref{sigma}). If the origin is an asymptotically stable equilibrium of $f(x)$ on a neighborhood $W$ of the origin, then for all $x \in W$, $L_o(x)$ is the unique smooth solution of
\[\label{Lo_hjb} \frac{\partial L_o}{\partial x}(x)f(x)+\frac{1}{2}h^{\tr}(x)h(x)=0,\quad L_o(0)=0 \]
under the assumption that (\ref{Lo_hjb}) has a smooth solution on $W$. Furthermore for all $x \in W$, $L_c(x)$ is the unique smooth solution of
\[\label{Lc_hjb} \frac{\partial L_c}{\partial x}(x)f(x)+\frac{1}{2} \frac{\partial L_c}{\partial x}(x)g(x)g^{\tr}(x)  \frac{\partial^{\tr} L_c}{\partial x}(x)=0, \quad L_c(0)=0\]
under the assumption that (\ref{Lc_hjb}) has a smooth solution $\bar{L}_c$ on $W$ and that the origin is an asymptotically stable equilibrium of $-(f(x)+g(x)g^{\tr}(x) \frac{\partial \bar{L}_c}{\partial x}(x))$ on $W$.
\end{theorem}

With the controllability and the observability functions on hand, the input-normal/output-diagonal realization of system~\eqref{sigma} can be computed by way of a coordinate transformation. More precisely,

\begin{theorem}\label{theorem_scherpen}\cite{scherpen_thesis,  scherpen_balancing}
Consider system~\eqref{sigma} under Assumption A and the assumptions in Theorem~\ref{thm:scherp1}. Then, there exists a neighborhood $W$ of the origin and coordinate transformation $x=\varphi(z)$ on $W$ converting  the energy functions  into the form
$L_c(\varphi(z))=\frac{1}{2}z^{\tr}z$  and $L_o(\varphi(z))=\frac{1}{2}\sum_{i=1}^nz_i^2\sigma_i(z_i)^2,$
where $\sigma_1(x) \ge \sigma_2(x) \ge \cdots \ge \sigma_n(x)$. The functions $\sigma_i(\cdot)$ are called {\em Hankel singular value functions}.
\end{theorem}

Analogous to the linear case, the system's states can be sorted in order of importance by sorting
the singular value functions, and reduction proceeds by removing the least important states.

In the above framework for balancing of nonlinear systems, one needs to solve (or numerically evaluate) the PDEs (\ref{Lo_hjb}), (\ref{Lc_hjb}) and compute the coordinate change $x=\varphi(z)$, however there are no systematic methods or tools for solving these problems. Various approximate solutions based on Taylor series expansions have been proposed~\cite{krener1,krener2,fujimoto}. Newman and Krishnaprasad~\cite{newman} introduce a statistical approximation based on exciting the system with white Gaussian noise and then computing the balancing transformation using an algorithm from differential topology. As mentioned earlier, an essentially linear empirical approach was proposed in~\cite{lall}. In this paper, we combine aspects of both data-driven approaches and analytic approaches by carrying out linear balancing of nonlinear control systems in a suitable RKHS.

\section{Empirical Balancing of Nonlinear Systems in RKHS}\label{sec:empirical_gramians}
We consider a general nonlinear system of the form
\[\label{eqn:nlsys}
\left\{\begin{array}{rcl} \dot{x}&=&f(x,u)\\ y&=&h(x) \end{array}\right.
\]
with $x \in \RR^n$, $u \in \RR^m$, $y \in \RR^p$, $f(0,0)=0$, and $h(0)=0$. Let ${\cal R}(x_0)=\{x' \in \RR^n: \exists\, u \in L_{\infty}(\RR,\RR^m) \;\;\mbox{and}\;\; \exists\, T \in [0,\infty) \;\; \mbox{such that}\; \;  x(0)=x_0\;\; \mbox{and}\;\; x(T)=x' \}$ be the reachable set from the initial condition $x(0)=x_0$.

{\it Hypothesis H:}\footnote{Let us note here that this assumption is similar to Assumption A above and that  we made it mainly out of convenience but is not necessary as illustrated in the 2D and 7D examples below. } The system (\ref{eqn:nlsys})  is zero-state observable, its linearization around the origin is controllable, and the origin of $\dot{x}=f(x,0)$ is asymptotically stable.

We treat the problem of estimating the observability and controllability Gramians as one of estimating an integral operator from data in a reproducing kernel Hilbert space (RKHS)~\cite{AronRKHS}, cf. Appendix for definition and key results on RKHSes. \emph{Our approach hinges on the key modeling assumption that the nonlinear dynamical system can be embedded in an appropriate high (or possibly infinite) dimensional RKHS where the method of linear balancing can be applied}.  More precisely, we will essentially assume that there is an RKHS ${\cal H}$   and   maps 
${\Phi, \Psi}: \RR^n \rightarrow {\cal H}; x \mapsto {\cal H}$ such that controllability and observability energies of the nonlinear system (\ref{eqn:nlsys}) are ``linearized'', i.e. that in ${\cal H}$ they have an expression similar  to the one in the linear case (\ref{lin_energies}) and, therefore, can be written as 
\[L_c(x) \approxeq\frac{1}{2}{ \Phi}^T(x) {\mathbb W}_c^{-1} { \Phi}(x), \quad  L_o(x)\approxeq \frac{1}{2}{ \Psi}^T(x) {\mathbb W}_o { \Psi}(x), \]
with $ {\mathbb W}_c, {\mathbb W}_o \in \RR^{N \times N}$, $N \gg n $, are very large dimensional matrices\footnote{If the system is affine in the input, we will use the PDEs (\ref{Lo_hjb}) and (\ref{Lc_hjb}) to find ${\Phi}$,  ${ \Psi}$, ${\mathbb W}_c$ and  ${\mathbb W}_o$. We leave such analysis for future work.}.   

By ``linearization" here, we mean  \emph{mapping the nonlinear system in a higher dimensional Hilbert space where linear theory can be applied}. To illustrate this point \cite{smola_book}, consider a polynomial in $\RR$, $p(x)=\alpha+\beta x + \gamma x^2$ where $\alpha$, $\beta$, and $\gamma$ are reals. If we consider the map $\Phi: \RR \rightarrow \RR^3$ defined as $\Phi(x)=[1 \; x \; x^2]^T$ then $p(x)= {\boldsymbol \alpha}\cdot [1 \; x \; x^2] ^T= {\boldsymbol \alpha}\cdot\Phi(x)$ is an affine polynomial in the variable $\Phi(x)$. 

Another example to illustrate our thought process is the one of Support Vector Machines (SVMs) that we referred to in the introduction. More precisely, consider the problem of classifying points in a data set $D=((x_1,y_1),\cdots,(x_n,y_n))$ with $x_i \in X$ with $X$ a set and $y_i=\pm 1$, i.e.  trying to find  $w \in \RR^d$ with $||w||_2=1$ and  $b \in \RR$ such that $\langle w,x_i \rangle +b >0$ for all $i$ with $y_i=+1$, and $\langle w,x_i \rangle+b <0$ for all $i$ with $y_i=-1$, i.e. that the linear hyperplane characterized by $(w,b)$ perfectly separates the set $D$ into two groups of data points, the ones with $y_i=+1$ and the ones with $y_i=-1$. Sometimes, it is possible to find such an hyperplane but, in general, finding a linear hyperplane that perfectly separates a given data set $D$ is not always possible and finding $(w,b)$ will not be possible. To solve the classification problem, the SVM algorithm maps the input data $(x_1,\cdots,x_n)$ into a (possibly infinite-dimensional) Hilbert space ${\cal H}$, the so-called \emph{feature space}, by a typically nonlinear map $\Phi:X \rightarrow {\cal H}$ called the \emph{feature map}. Then one looks for a linear hyperplane that separates the data $((\Phi(x_1),y_1),\cdots,(\Phi(x_n),y_n))$, i.e. one looks for $(w,b)$ in  ${\cal H}$. When this is possible, the data in $D$ will be classified in two categories, the ones with $y_i=+1$ and the ones with $y_i=-1$, but  the separating curve for $D$ will not be a linear hyperplane in the original space (even if it is a linear hyperplane in ${\cal H}$). An important property of SVMs is that for every dataset D without contradicting points, i.e. $(x_i,y_i)$ and $(x_j,y_j)$ with $x_i=x_j$ and $y_i \ne y_j$, there exists a feature map  that allows the  perfect separation by a hyperplane in the feature space \cite{steinwart_svms}.

 Our aim is to generalize this way of thinking to nonlinear dynamical systems, i.e. given a problem for a nonlinear dynamical system, we map the state variables by a typically nonlinear map $\Phi:X \rightarrow {\cal H}$  where ${\cal H}$ is  (possibly infinite-dimensional) Hilbert space  in which the computations become simpler\footnote{One could also think of the methods in Quantum Mechanics where one constructs a Hilbert space  from measurements in order to perform computations.}. In our case, ``simpler'' means applying Linear Theory. In this paper, we will focus on the problem of approximation of nonlinear control systems by applying the method of linear balancing in RKHSes. We leave other applications for future work.

Covariance operators in RKHSes and their empirical estimates are the objects of primary importance and contain the information needed to perform model reduction. In particular, the (linear) observability and controllability Gramians  are estimated and diagonalized in the RKHS, but capture nonlinearities in the original state space. The reduction approach we propose adapts ideas from kernel PCA (KPCA)~\cite{KPCA:98} and is driven by a set of simulated or sampled system trajectories, extending and generalizing the work of Moore~\cite{moore} and Lall et al.~\cite{lall}. 

Our method works quite well since the controllability and the observability energies in the linear case can be expressed as inner products (\ref{lin_energies}) and \emph{working in RKHSes allows to find nonlinear versions of linear algorithms that can be expressed in terms of inner products} (this is the so-called \emph{kernel trick} in Learning Theory, see Appendix A for more explanations). Hence the empirical Gramians defined below for nonlinear systems can be viewed as reasonable approximations of the controllability and observability energies for nonlinear systems.

In the development below we lift state vectors of the system (\ref{eqn:nlsys}) into a Hilbert space ${\cal H}$, i.e. we consider a mapping $\Phi: \RR^n \rightarrow {\cal H}$ and analyze the nonlinear system whose state is $\Phi(x)$.



\subsection{Empirical Gramians in RKHS}
Following~\cite{moore}, we estimate the controllability Gramian by exciting each coordinate
of the input with impulses\footnote{This is not a limitation of our approach, other input signals can be used such as a white gaussian noise, cf. \cite{acc2012} for preliminary results.} while setting $x_0 = 0$. One can also further excite using rotations of impulses
as 
suggested in~\cite{lall}, however for simplicity we consider only the original signals proposed
in~\cite{moore}. Let $u^i(t) = \delta(t)e_i$ be the $i$-th excitation
signal, and let $x^i(t)$ be the corresponding response of the system. Form the matrix \[\label{resp1} X(t) = \bigl[x^1(t) ~\cdots~ x^m(t)\bigr] \in \RR^{n\times m},\] so that $X(t)$ is seen as a data matrix with
column observations given by the respective responses $x^i(t)$. Then $W_c\in\RR^{n\times n}$ is given by
\[\label{wc_empirical}
W_c = \frac{1}{m}\int_0^{\infty}X(t)X(t)^{\tr} dt.
\]
We can approximate this integral by sampling the matrix function $X(t)$ within a finite time interval $[0,T]$ assuming the regular partition $\{t_i\}_{i=1}^N, t_i = (T/N)i$. This leads to the empirical controllability Gramian
\[\label{wc_hat}
\widehat{W}_c = \frac{T}{mN}\sum_{i=1}^N X(t_i)X(t_i)^{\tr} .
\]

As described in~\cite{moore}, the observability Gramian is estimated by
fixing $u(t) = 0$, setting $x_0 = e_i$ for $i=1,\ldots,n$, and measuring the corresponding system output responses $y^i(t)$. As before, assemble the responses into a matrix $Y(t) = [y^1(t) \cdots y^n(t)]\in \RR^{p\times n}$. The observability Gramian $W_o\in\RR^{n\times n}$ and its empirical
counterpart $\widehat{W}_o$ are given by
\[\label{wo_empirical}
W_o = \frac{1}{p}\int_0^{\infty}Y(t)^{\tr}Y(t) dt\]
and
\[\label{wo_hat} \widehat{W}_o = \frac{T}{pN}\sum_{i=1}^N \widetilde{Y}(t_i)\widetilde{Y}(t_i)^{\tr}
\]
where $\widetilde{Y}(t) = Y(t)^{\tr}$.
The matrix $\widetilde{Y}(t_i)\in\RR^{n\times p}$ can be thought of as a data matrix with column observations
\begin{equation}\label{eqn:obs_data}
d_j(t_i) = \bigl(y_j^1(t_i), \ldots, y_j^n(t_i)\bigr)^{\!\tr} \in\RR^n,\quad j=1,\ldots,p, \,\,i=1,\ldots, N
\end{equation}
so that $d_j(t_i)$ corresponds to the response at time $t_i$ of the single output coordinate $j$ to each of the (separate) initial conditions $x_0=e_k, k=1,\ldots,n$. This convention will lead to greater clarity in the steps that follow.

\subsection{Model Order Reduction Map}
The method we propose consists, in essence, of collecting samples and then performing a process similar to ``simultaneous principal components analysis'' on the controllability and observability Gramian estimates in the (same) RKHS. 


As mentioned above, given a choice of the kernel $K$ defining a RKHS $\cH$, principal components in the feature space can be computed implicitly in the original input space using $K$. It is worth emphasizing however that we will be co-diagonalizing {\em two} Gramians in the feature space by way of a {\em non-orthogonal} transformation; the process bears a resemblance to (K)PCA, and yet is distinct. Indeed the favorable properties associated with an orthonormal basis are no longer available, the quantities we will in practice diagonalize are different, and the issue of data-centering must be considered with some additional care.

First note that the empirical controllability Gramian $\widehat{W}_c$ can be viewed as the sample covariance of a collection of $N\cdot m$ vectors, scaled by $T$
\begin{equation}\label{eqn:contgram_vectors}
\widehat{W}_c = \frac{T}{mN}\sum_{i=1}^N X(t_i)X(t_i)^{\tr} =
\frac{T}{mN}\sum_{i=1}^N\sum_{j=1}^m x^j(t_i)x^j(t_i)^{\tr}
\end{equation}
where $X(t)$ is defined in (\ref{resp1}) and the observability Gramian can be similarly viewed as the sample covariance of a collection of $N\cdot p$ vectors
\begin{equation}\label{eqn:obsgram_vectors}
\widehat{W}_o = \frac{T}{pN}\sum_{i=1}^N\sum_{j=1}^p d_j(t_i)d_j(t_i)^{\tr}
\end{equation}
where the $d_j$ are defined in Equation~\eqref{eqn:obs_data}.

We can thus consider three quantities of interest: 
\begin{itemize} 
\item The \emph{controllability kernel matrix} $K_c\in\RR^{Nm\times Nm}$ of kernel
products 
\[\label{controllability_kernel_matrix} (K_c)_{\mu\nu} = K(x_\mu, x_\nu) = \scal{\Phi(x_\mu)}{\Phi(x_\nu)}_{\cal H}\] 
for $\mu,\nu=1,\ldots,Nm$ where we have re-indexed the set of vectors  $\{x^{j}(t_i)\}_{i,j} = \{x_{\mu}\}_{\mu}$ to use a single linear index.
\item The \emph{observability kernel matrix} $K_o\in\RR^{Np\times Np}$,
\[\label{eqn:obs_kern_mat}
(K_o)_{\mu\nu} = K(d_\mu, d_\nu) = \scal{\Phi(d_\mu)}{\Phi(d_\nu)}_{\cal H}
\]
for $\mu,\nu=1,\ldots,Np$, where we have again re-indexed the set $\{d_j(t_i)\}_{i,j}=\{d_\mu\}_{\mu}$ for simplicity.
\item The \emph{Hankel kernel matrix} $K_{o,c}\in\RR^{Np\times Nm}$,
\[\label{Hankel_kernel_matrix} (K_{o,c})_{\mu\nu} = K(d_\mu, x_\nu) = \scal{\Phi(d_\mu)}{\Phi(x_\nu)}_{\cal H}\]
for $\mu=1,\ldots,Np$, $\nu=1,\ldots,Nm$.
\end{itemize}
We have chosen the suggestive terminology ``Hankel kernel matrix'' above because the square-roots of the  nonzero eigenvalues of the matrix $K_{o,c}K_{o,c}^{\tr}$ are the empirical Hankel singular values of the system mapped into feature space\footnote{The relation between the singular values and the eigenfunctions of the Hankel operator for the nonlinear system given by ${\mathbb H}$ in (\ref{hankel_nonlinear}) and the extension of (\ref{hankel_lin1})-(\ref{hankel_lin2}) to the nonlinear setting using  the empirical Hankel singular values and eigenvectors of $K_{o,c}K_{o,c}^{\tr}$  is an open problem that we leave for future work. }, where we assume that the method of linear balancing can be applied. This assertion will be proved immediately below. Note that ordinarily, $Nm,Np\gg n$ and $K_c,K_o$ will be rank deficient.


Before proceeding we consider the issue of data centering in feature space. PCA and kernel PCA assume that the data have been centered in order to make the problem translation invariant. In the setting considered here, we have two distinct sets of data: the observability samples and the controllability samples. A reasonable centering convention centers the data in each of these datasets separately. Let $\bPsi$ denote
the matrix whose columns are the observability samples mapped into feature space by the feature map $\Phi$, and let $\bPhi$ be the matrix similarly built from the feature space representation of the controllability samples. Then \[K_o = \bPsi^{\tr}\bPsi, \quad K_c=\bPhi^{\tr}\bPhi, \quad \mbox{and} \quad K_{o,c} = \bPsi^{\tr}\bPhi. \]
The above equation reduces to (\ref{wc_hat}) and (\ref{wo_hat}) in the linear case. In fact, when $\Phi(x)=x$, 
$(K_c)_{\mu \nu}=\langle x_{\mu}, x_{\nu}\rangle_{\RR^n}$, $(K_o)_{\mu \nu}=\langle d_{\mu}, d_{\nu}\rangle_{\RR^n},$ $(K_{o,c}) =\langle d_{\mu}, x_{\nu}\rangle_{\RR^n}$.

 Assume for the moment that there are $M$ observability data samples and $N$ controllability samples, and let
$\bbone_N, \bbone_M$ denote the length  $N$, $M$ vectors of all ones, respectively. We can define centered versions of the feature space data matrices $\bPhi,\bPsi$ as
\[
\widetilde{\bPhi} = \bPhi - \mu_c\bbone^{\tr}_N, \qquad \widetilde{\bPsi} = \bPsi - \mu_o\bbone^{\tr}_M
\]
where $\mu_c:=N^{-1}\bPhi\bbone_N$ and $\mu_o:=M^{-1}\bPsi\bbone_M$. We will need two centered quantities in the development below.  Let us note here that, in practice, we do not need to compute $\mu_c$ and $\mu_o$ as detailed above. Moreover, there is no need to compute the embedding $\Phi$ since all computations are done in terms of the kernels.

The first centered quantity we consider is the centered version of $K_{o,c}$, namely $\widetilde{K}_{o,c}=\widetilde{\bPsi}^{\tr}\widetilde{\bPhi}$. Although one cannot compute $\mu_c,\mu_o$ explicitly from the data, we can compute $\widetilde{K}_{o,c}$ by observing that
\begin{align}\label{eqn:koc_centering}
\widetilde{K}_{o,c} &= \bigl(\bPsi - \mu_o\bbone^{\tr}_M\bigr)^{\tr}\bigl(\bPhi - \mu_c\bbone^{\tr}_N\bigr)\nonumber\\
 &= K_{o,c} - \tfrac{1}{N}K_{o,c}\bbone_N\bbone^{\tr}_N - \tfrac{1}{M}\bbone_M\bbone^{\tr}_MK_{o,c} 
    + \tfrac{1}{NM}\bbone_M\bbone_M^{\tr}K_{o,c}\bbone_N\bbone_N^{\tr}.
 \end{align}
The second quantity we'll need is a centered version of the {\em empirical observability feature map}
\[\label{eqn:emp_kmap}
\mathbf{k}_o(x):=\bPsi^{\tr}\Phi(x) = \bigl(K(x,d_1),\ldots,K(x,d_M)\bigr)^{\tr}
\]
where $x\in\RR^d$ is the state variable and the observability samples $\{d_j\}$ are again indexed by a single variable as in Equation~\eqref{eqn:obs_kern_mat}. Centering follows reasoning similar to that of the Hankel kernel matrix immediately above:
\begin{align}\label{eqn:ko_centering}
\widetilde{\mathbf{k}}_o(x) &= \bigl(\bPsi - \mu_o\bbone^{\tr}_M\bigr)^{\tr}\bigl(\Phi(x) - \mu_c\bigr) \nonumber\\
&= \mathbf{k}_o(x) - \tfrac{1}{N}K_{o,c}\bbone_N - \tfrac{1}{M}\bbone_M\bbone^{\tr}_M\mathbf{k}_o(x) + 
\tfrac{1}{NM}\bbone_M\bbone_M^{\tr}K_{o,c}\bbone_N.
\end{align}
Note: {\em Throughout the remainder of this paper we will drop the special notation $\widetilde{K}_{o,c}$, $ \widetilde{\mathbf{k}}_o(x)$ and assume that $K_{o,c}, \mathbf{k}_o(x)$ are centered appropriately.}

With the quantities defined above, we can co-diagonalize the empirical Gramians (balancing) and reduce the dimensionality of the state variable (truncation) in feature space by carrying out calculations in the original data space. As we are assuming that the method of linear balancing applies in the feature space, the order of the model can be reduced by discarding small Hankel values $\{\Sigma_{ii}\}_{i=q+1}^n$, and projecting onto the subspace associated with the first $q<n$ largest eigenvalues. The following key result describes this process:

\begin{theorem}[Balanced Reduction in Feature Space]\label{thm:redmap_thm} 
Consider the nonlinear control system (\ref{eqn:nlsys}) and its responses (\ref{resp1}) and (\ref{eqn:obs_data}) to impulses from the input and initial condition, respectively. Let $K$ be a Mercer kernel,   $K_{o,c}$ be the Hankel kernel matrix   defined in (\ref{Hankel_kernel_matrix}), and $K_{o,c}K_{o,c}^{\tr}=V\Sigma^2V^{\tr}$ be its SVD decomposition with $\Sigma=\mbox{diag}\{\sigma_1,\cdots \sigma_{Np} \}$ and $\sigma_i \ge \sigma_{i+1}$ for $i=1,\cdots Np$. Also consider  the controllability kernel matrix $K_c$ (\ref{controllability_kernel_matrix}), the observability kernel matrix  $K_o$ (\ref{eqn:obs_kern_mat}).

If there is a spectral gap in $\Sigma$, i.e. there is $q$ such that  $\sigma_q >> \sigma_{q+1}$ then balanced reduction in the RKHS can be accomplished by applying the state-space reduction map
$\Pi:\RR^n\to\RR^q$ given by
\begin{equation}\label{eqn:pi_map}
\Pi(x) = T_q^{\tr}\mathbf{k}_o(x), \quad x\in\RR^n
\end{equation}
where $T_q = V_q\Sigma^{-1/2}_q$, $V_q$ are the eigenvectors that correspond to the largest $q$ Hankel singular values.
$\mathbf{k}_o(x)$ is the empirical observability feature map (\ref{eqn:ko_centering}). 
\end{theorem}
\begin{proof}
We assume the data have been centered in feature space. Let $\bPhi$ be a matrix with columns $\bigl\{\Phi(x^j(t_i))\bigr\}, i=1,\ldots,N, j=1,\ldots,m$, so that \[X=\bPhi\bPhi^{\tr},\] is the feature space controllability Gramian counterpart to Equation~\eqref{eqn:contgram_vectors}. Similarly,
let $\bPsi$ be a matrix with columns $\bigl\{\Phi(d_j(t_i))\bigr\}, i=1,\ldots,N, j=1,\ldots,p$, so that
\[Y=\bPsi\bPsi^{\tr},\] is the feature space observability Gramian counterpart to Equation~\eqref{eqn:obsgram_vectors}. Since by definition $K(x,y)=\scal{\Phi(x)}{\Phi(y)}_{\cF}$, we also have that $K_c=\bPhi^{\tr}\bPhi$ and $K_o=\bPsi^{\tr}\bPsi$. In general the Gramians $X,Y$ are infinite dimensional whereas the kernel matrices $K_c, K_o$ are necessarily of finite dimension.

We now carry out linear balancing on $(X,Y)$ in the feature space (RKHS).  First, take the SVD of $X^{1/2}\bPsi$ so that
\begin{align}
U\Sigma V^{\tr} &= X^{1/2}\bPsi \label{eqn:firstSVD} \\
U\Sigma^2U^{\tr} &= (X^{1/2}\bPsi)(X^{1/2}\bPsi)^{\tr} = X^{1/2}YX^{1/2}  \label{eqn:secSVD}\\
 V\Sigma^2V^{\tr}&= (X^{1/2}\bPsi)^{\tr}(X^{1/2}\bPsi) = \bPsi^{\tr}\bPhi\bPhi^{\tr}\bPsi = K_{o,c}K_{o,c}^{\tr}. \label{eqn:thSVD}
\end{align}
The last equality in Equation~\eqref{eqn:secSVD} follows since $X$ is symmetric and therefore $X^{1/2}$ is too. The linear balancing transformation is then given by \[{\cal T}=\Sigma^{1/2}U^{\tr}X^{-1/2},\] and one can readily verify that ${\cal T} X{\cal T}^{\tr}={\cal T}^{-\tr}Y{\cal T}^{-1}=\Sigma$. Here, inverses should be interpreted as peudo-inverses when appropriate\footnote{Such as in the case of $X^{-1/2}$ when the number of data points is less than the dimension of the RKHS.}. From Equations~\eqref{eqn:firstSVD}-\eqref{eqn:thSVD}, we see that $U^{\tr}=\Sigma^{-1}V^{\tr}\bPsi^{\tr}X^{1/2}$ and thus ${\cal T}=\Sigma^{-1/2}V^{\tr}\bPsi^{\tr}$. We can project an arbitrary mapped data point $\Phi(x)$ onto the (balanced) ``principal'' subspace of dimension $q$ spanned by the first $q$ rows of ${\cal T}$ by computing
\[\label{eqn:feat_space_proj}
{\cal T}_q\Phi(x) = \Sigma_q^{-1/2}V_q^{\tr}\bPsi^{\tr}\Phi(x) = \Sigma_q^{-1/2}V_q^{\tr}\mathbf{k}_o(x)
\]
where $\mathbf{k}_o(x):=\bPsi^{\tr}\Phi(x)$ is the empirical observability feature map, recalling that $V_q$ is the matrix formed by taking the top $q$ eigenvectors of $K_{o,c}K_{o,c}^{\tr}$ by Equation~\eqref{eqn:thSVD}.
\end{proof}

We note that square roots of the non-zero eigenvalues of $K_{o,c}K_{o,c}^{\tr}$ are exactly the Hankel singular values of the system mapped into the feature space, under the assumption of linearity in the feature space. This can be seen by noting that $\lambda_+(YX)=\lambda_+(X^{1/2}YX^{1/2})=\lambda_+(K_{o,c}K_{o,c}^{\tr})$, where $\lambda_+(\cdot)$ refers to the non-zero eigenvalues of its argument. In practice, we compute the largest eigenvalues of $K_{o,c}K_{o,c}^{\tr}$ instead of performing its SVD.

In Section~\ref{sec:closed_sys} below we show how to use the nonlinear reduction map~\eqref{eqn:pi_map} to realize a closed, reduced order system which can approximate the original system to a high degree of accuracy.



\section{Closed Dynamics of the Reduced System}\label{sec:closed_sys}
Given the nonlinear state space reduction map $\Pi:\RR^n\to\RR^q$, a remaining challenge is to
construct a corresponding (reduced) dynamical system on the reduced state space which well approximates
the input-output behavior of the original system on the original state space. Setting $x_r = \Pi(x)$ and
applying the chain rule,
\begin{equation}\label{eqn:exact_closed}
\dot{x}_r = \left.\bigl(J_{\Pi}(x)f(x,u)\bigr)\right|_{x = \Pi^{\dag}(x_r)}
\end{equation}
where $\Pi^{\dag}$ refers to an appropriate notion (to be defined) of the inverse of $\Pi$.
However we are faced with the difficulty that the map $\Pi$ is not in general surjective
(even if $q=n$), and moreover one cannot guarantee that an arbitrary point in the RKHS has a non-empty
preimage under $\Phi$~\cite{mika98}. We propose an
approximation scheme to get around this difficulty: The dynamics $f$ will be approximated by an
element of an RKHS {\em defined on the reduced state space}. When $f$ is assumed to be known explicitly it can be approximated to a high degree of accuracy. An approximate, least-squares notion of $\Pi^{\dag}$ will be given to first or second order via a Taylor series expansion, but only where it is strictly needed -- and at the last possible moment -- so that a first or second order approximation will not be as crude as one might suppose. We will also consider, as an alternative, a direct approximation of $J_{\Pi}(\Pi^{\dag}(x_r))$ which takes into account further properties of the reproducing kernel as well as the fact that the Jacobian is to be evaluated at $x = \Pi^{\dag}(x_r)$ in particular. In both cases, the important ability of the map $\Pi$ to capture strong nonlinearities will not be significantly diminished.

\subsection{Representation of the dynamics in RKHS}\label{sec:f_rkhs}
The vector-valued map $f:\RR^{n}\times\RR^m\to\RR^n$ can be approximated by  composing a set of $n$ regression functions (one for each coordinate) $\hat{f}_i:\RR^{q\times m}\to\RR$ in an RKHS using the representer theorem in Appendix A, with the reduction map $\Pi$. It is reasonable to expect that this approximation will be better than directly computing $f(\Pi^{\dag}(x_r),u)$ using, for instance, a Taylor expansion approximation for $\Pi^{\dag}$ which may ignore important nonlinearities at a stage where crude approximations must be avoided. Let us note here that, for this approximation part, we are using the representer theorem as described in Appendix A and that the RKHS we use is not necessarily the same as the one where we have been performing balancing as in the previous section.

Let $\tilde{x}=\Pi(x)$ denote a reduced state  variable, and concatenate the input examples $\tilde{x}_j:=\Pi(x_j)\in\RR^q,u_j\in\RR^m$ so that $z_j=(\tilde{x}_j,u_j)\in\RR^{q\times m}$, and $\{\left(f_i(x_j,u_j), z_j\right)\}_{j=1}^{\ell}$ is a set of input-output training pairs describing the $i$-th coordinate of the map $(\tilde{x},u)\mapsto f(x,u)$. The training examples should characterize
``typical'' behaviors of the system, and can even re-use those trajectories simulated in response to impulses for estimating the Gramians above. We will seek the function $\hat{f}_i\in\cH$ which minimizes
$$
\sum_{j=1}^{\ell}\bigl(\hat{f}_i(z_j) - f_i(x_j,u_j)\bigr)^2 + \lambda_i\|\hat{f}_i\|^2_{\cH}
$$
where $\lambda_i$ here is a regularization parameter. We have chosen the square loss, however other suitable loss functions may be used (cf. \cite{Wahba, pontil, cucker, schaback_survey}).
As mentioned in section \ref{sec:defins}, equation (\ref{learning1}),  $\hat{f}_i$ takes the form 
\[\label{f_hat} \hat{f}_i(z) =\hat{f}_i(\Pi(x),u)= \sum_{j=1}^{\ell}c_j^iK^f(z,z_j), i=1,\ldots,n,\] 
where $K^f$ defines the RKHS $\cH_f$ (and is unrelated to $K$ used to estimate the Gramians). Note that although our notation takes the RKHS for each coordinate function to be the same, in general this need not be true: different kernels may be chosen for each function. Here the $\{c_j^i\}|_{(i,j)=(1,1)}^{(i,j)=(n,\ell)}$ comprise a set of coefficients found using the regularized least squares (RLS) algorithm and satisfying the algebraic equation (\ref{learning2}) with ${\bf s}=\{\left(f_i(x_j,u_j), z_j\right)\}_{j=1}^{\ell}$  as training examples. The kernel family and any hyper-parameters can be chosen by cross-validation.
For notational convenience we will further define the vector-valued empirical feature map
\begin{equation*}
\bigl(\bk^{f}(\tilde{x},u)\bigr)_i:= K^f\bigl( (\tilde{x},u), z_i \bigr)
\end{equation*}
for $i=1,\ldots,\ell$. In this notation $\hat{f}_i\bigl(\Pi(x),u\bigr) = \bc_i^{\tr}\bk^{f}(\tilde{x},u)$
where $(\bc_i)_j = c_j^i$.
%

A broad class of systems seen in the literature~\cite{scherpen_thesis} are also characterized by separable dynamics of the form $\dot{x} = f(x) + \sum_{i=1}^mg_i(x)u_i$. In this case, one can
estimate the functions $f$ and $g_i$ from examples $\{(\Pi(x_j),f(x_j))\}_j$ and $\{(\Pi(x_j),g(x_j))\}_j$.

\subsection{Approximation of the Jacobian Contribution}\label{sec:jacobian_app}
We turn to approximating the component $J_{\Pi}\bigl(\Pi^{\dag}(x_r)\bigr)$ appearing in Equation~\eqref{eqn:exact_closed}.

\subsubsection{Inverse-Taylor Expansion}
A simple solution is to compute a low-order Taylor expansion of
$\Pi$ and then invert it using the Moore-Penrose pseudoinverse to obtain the approximation.
For example, consider the first order expansion $\Pi(x) \approx \Pi(a) + J_{\Pi}(a)(x-a)$. Then we
can approximate $\Pi^{\dag}(x_r)$ (in the first-order, least-norm sense) as
\[
\widehat{\Pi}^{\dag}(x_r) := \bigl(J_{\Pi}(a)\bigr)^{\dagger}(x_r - \Pi(a)) + a .
 \]
We may start with $a=x_0$, but periodically update the expansion in different regions of the dynamics if desired. A good expansion point could be the estimated preimage of $x_r(t)$ returned by the algorithm proposed in~\cite{Kwok:ICML:03}. If $\widetilde{\mathbf{k}}_o(x)$ is the centered version of the length $M$ vector $\mathbf{k}_o(x)$ defined by~\eqref{eqn:emp_kmap},  and since $\Pi(x)=T_q^T\widetilde{{\mathbf k}}_o(x)$, then
\[\label{j_pi_x}
J_{\Pi}(x) = \frac{\partial \Pi(x)}{\partial x} = T_q^{\tr}\bigl(I - \tfrac{1}{M}\bbone_M\bbone_M^{\tr}\bigr)\frac{\partial \mathbf{k}_o(x)}{\partial x}
\]
where $\bbone_M$ is the length $M$ vector of all ones.  $J_{\Pi}(x) $ is a matrix in $\RR^{q \times n}$ since $T_q=V_q \Sigma_q^{-\frac{1}{2}} \in \RR^{M \times q}$ with $V_q \in \RR^{M \times q}$, $K_{oc} \in  \RR^{Np \times Nm}$, $\Sigma_q \in \RR^{q \times q}$,  $(I-\frac{1}{M} \mathbf{1_M}  \mathbf{1_M}^T )\in  \RR^{M \times M}  $, $k_o(x) \in \RR^{M \times 1}$, $\frac{\partial k_o(x)}{\partial x} \in \RR^{M \times n} $.
 An example calculation of $\bigl(\partial_{x}\mathbf{k}_o(x)\bigr)_i = \partial_{x}K(x,d_i)$ in the case of a polynomial kernel
is given in the section immediately below.

\subsubsection{Exploiting Kernel Properties}
For certain choices of the kernel $K$ defining the Gramian feature space $\cH$, one can exploit
the fact that $K_x$ and its derivative bear a special relationship, and potentially improve the estimate for $J_{\Pi}(\Pi^{\dag}(x_r))$. Perhaps the most commonly used off-the-shelf kernel families are the polynomial and Gaussian families. For any two kernels with hyperparameters $p$ and $q$ (respectively) in one of these classes, we have that $K_p = (K_q)^{p/q}$. We'll consider the polynomial kernel of degree $d$,
$K_d(x,y):=(1+\scal{x}{y})^d$ in particular; the Gaussian case can be derived using similar reasoning. For a polynomial kernel we have that
$$
\pfrac{K_d(x,y)}{x} = dK_{d-1}(x,y)y^{\tr} = d\bigl(K_d(x,y)\bigr)^{\tfrac{d-1}{d}}y^{\tr}.
$$
Recalling that $K_d(x,y)=\scal{\Phi(x)}{\Phi(y)}_{\cH}$ and $x_r=\Pi(x)$ given by (\ref{eqn:pi_map}), if $\Pi$ were invertible then we would have
$$
\left.\pfrac{K_d(x,y)}{x}\right|_{x=\Pi^{-1}(x_r)} =
d\scal{(\Phi\circ\Pi^{-1})(x_r)}{\Phi(y)}^{\tfrac{d-1}{d}}y^{\tr}.
$$
The map $\Pi$ is not injective however, and in addition the fibers of $\Phi$ may be potentially empty, so we must settle for an approximation. It is reasonable then to {\em define}
$(\Phi\circ\Pi^{\dag})(x_r)$ as the solution to the convex optimization problem
\begin{equation}\label{eqn:inv-min}
\begin{aligned}
& \underset{z\in\cH}{\min} & & \nor{z}_{\cH} \\
& \text{subj. to} & & \nor{M_qz - x_r}_{\RR^k} = 0
\end{aligned}
\end{equation}
where $M_q:\cH\to\RR^k$ is defined as in Equation~\eqref{eqn:feat_space_proj}.
If a point $z\in\cH$ has a pre-image in $\RR^n$ this definition is consistent with composing
$\Phi$ with the formal definition $\Phi^{-1}(z) = \{x\in\RR^n~|~\Phi(x)=z\}$ and noting that in this case
$\Pi\circ\Phi^{-1} = M_q(\Phi\circ\Phi^{-1}) =M_qz$. Furthermore, a trajectory $x_r(t)$
of the closed dynamical system on the reduced statespace need not (and may not) have a counterpart in the
original statespace by virtue of the way in which $\Pi^{\dag}$ is used in our formulation of the reduction map and corresponding reduced dynamical system.

One will recognize that the solution $z^*$ to~\eqref{eqn:inv-min} is just the Moore-Penrose pseudoinverse $z^* = M_q^{\dagger}x_r$. Inserting this solution into the feature map representation of a kernel $K$ gives the following definition for $K(\Pi^{\dag}(x_r),y)$:
\begin{align*}
K(\Pi^{\dag}(x_r),y) &= \scal{(\Phi\circ\Pi^{\dag})(x_r)}{\Phi(y)}_{\cH} \\
 &= \scal{M_q^{\dagger}x_r}{\Phi(y)}_{\cH}  = \bigl\langle x_r, (M_q^{\tr})^{\dagger}\Phi(y)\bigr\rangle_{\RR^k} \\
 & = \bigl\langle x_r,(M_qM_q^{\tr})^{-1}M_q\Phi(y)\bigr\rangle \\
 & = \bigl\langle x_r,(M_qM_q^{\tr})^{-1}\Pi(y)\bigr\rangle \\
 & = \bigl\langle x_r,(T_q^{\tr}K_oT_q)^{-1}\Pi(y)\bigr\rangle \,
 \end{align*}
where the final equality follows applying Equations~\eqref{eqn:firstSVD}-\eqref{eqn:thSVD} and
$T_q$ is defined as in Theorem~\ref{thm:redmap_thm}.
Substituting into the derivative for a polynomial kernel $K=K_d$ gives
$$
\left.\pfrac{K_d(x,y)}{x}\right|_{x=\Pi^{\dag}(x_r)} =
d\bigl\langle x_r,(T_q^{\tr}K_oT_q)^{-1}\Pi(y)\bigr\rangle^{\tfrac{d-1}{d}}y^{\tr}
$$
which immediately gives an expression for $J_{\Pi}(\Pi^{\dag}(x_r))$ using (\ref{j_pi_x}):
\[\label{j_pi_x_poly}
 {J_{\Pi}(x)}|_{x=\Pi^{\dag}(x_r)}= T_q^{\tr}\bigl(I - \tfrac{1}{M}\bbone_M\bbone_M^{\tr}\bigr)
d\bigl\langle x_r,(T_q^{\tr}K_oT_q)^{-1}\Pi(d_i)\bigr\rangle^{\tfrac{d-1}{d}}d_i^{\tr}
\]

Note that this approximation is global in the sense that the $q\times q$ matrix inverse  $(T_q^{\tr}K_oT_q)^{-1}$ need only  be computed once\footnote{We use the word ``inverse'' loosely. In practice one would use a numerically stable method, such as an LU-factorization, which can be used to rapidly compute $A^{-1}b$ for fixed $A$ but many different $b$.}; no updating is required during simulation of the closed system.

\subsection{Reduced System Dynamics} \label{reduced_model}
Given an estimate $\hat{f}\bigl(\Pi(x),u\bigr)$ of $f(x,u)$ in the RKHS $\cH_f$ as in (\ref{f_hat}) and
a notion of $J_{\Pi}\bigl(\Pi^{\dag}(x_r)\bigr)$ from (\ref{j_pi_x}) and the section above, we can write down a closed dynamical system on the reduced statespace. We have
\begin{align}\label{eqn:approx_xr}
\dot{x}_r &\approx \left.\bigl(J_{\Pi}(x)\hat{f}(\Pi(x),u)\bigr)\right|_{x = \Pi^{\dag}(x_r)} \nonumber\\
&\approx \left.\bigl(J_{\Pi}(x)\bigr)\right|_{x = \Pi^{\dag}(x_r)}\bC^{\tr}\bk^{f}(x_r,u) \nonumber\\
&= T_q^{\tr}\bigl(I - \tfrac{1}{M}\bbone_M\bbone_M^{\tr}\bigr)\frac{\partial \mathbf{k}_o(x)}{\partial x}\bigg|_{x = \Pi^{\dag}(x_r)} \bC^{\tr}\bk^{f}(x_r,u)\nonumber\\
&:=T_q^{\tr}\bigl(I - \tfrac{1}{M}\bbone_M\bbone_M^{\tr}\bigr)J_{\mathbf{k}}\bigl(\Pi^{\dag}(x_r)\bigr) \bC^{\tr}\bk^{f}(x_r,u)
\end{align}
where $\bC$ is a matrix with the vectors $\bc_i$ as its rows, and $J_{\mathbf{k}}\bigl(\Pi^{\dag}(x_r)\bigr) :=\frac{\partial \mathbf{k}_o(x)}{\partial x}\big|_{x = \Pi^{\dag}(x_r)}$ is the Jacobian of the empirical feature map defined in Equation~\eqref{eqn:emp_kmap}. Here the expression 
$J_{\mathbf{k}}\bigl(\Pi^{\dag}(x_r)\bigr)$ should be interpreted as notation for either of the Jacobian approximations suggested in Section~\ref{sec:jacobian_app}.


Equation~\eqref{eqn:approx_xr} is seen to give a closed nonlinear control system
expressed solely in terms of the reduced variable \mbox{$x_r\in\RR^q$}:
\begin{equation}
 \label{reduced_system}
\left\{
\begin{aligned}
\dot{x}_r &=  {J_{\Pi}(x)}|_{x=\Pi^{\dag}(x_r)}\bC^{\tr}\bk^{f}(x_r,u)\\
\hat{y} &= \hat{h}(x_r)
\end{aligned}
\right.
\end{equation}
where the map $\hat{h}\circ\Pi$ modeling the output function $h:\RR^n\to\RR^p$ is estimated as described immediately below.
Although the ``true'' reduced system does not actually exist due to non-surjectivity of the
feature map $\Phi$, in many situations one can expect that the above system will capture the essential
input-output behavior of the original system. We leave a precise analysis of the error in the approximations
appearing in~\eqref{eqn:approx_xr} to future work\footnote{There are different sources of error in this process of approximation. A first one due to truncation in the RKHS which could be characterized by extending  (\ref{error_y_yr}) using the eigenvalues of  the Hankel kernel matrix (\ref{Hankel_kernel_matrix}). There is also the error in finding the pseudo inverse of $\Phi$ in (\ref{eqn:inv-min}), i.e. that it is tempting to write $||y-\hat{y}|| \le 2 \bigg(\sum_{j=q+1}^{Np} \sigma_j \bigg) ||u|| $ modulo the fact that $\Phi$ may not be surjective and that $\hat{y}$ in (\ref{reduced_system}) is not what we will end up observing as output for the reduced order system in $\RR^q$ and the fact that there is also the error in approximating $f(x,u)|_{x=\Pi^{\dag}(x_r)}$ by  $\bC^{\tr}\bk^{f}(x_r,u)$ that might be computed using results from \cite{smale_approximation_error}  for a given $u$ as illustrated in section  \ref{error_estimates} below.}.

\subsection{Outputs of the Reduced System} \label{output}
Analogous to the case of the dynamics $f$, we are faced with two possibilities for approximating
\[\label{output_map} y = h\bigl(\Pi^{\dag}(x_r)\bigr).\] We can apply a crude Taylor series approximation to estimate $\Pi^{\dag}$ and therefore $h\bigl(\Pi^{\dag}(x_r)\bigr)$, or as in Section~\ref{sec:f_rkhs}, we can estimate a map \[(\hat{h}\circ\Pi):\begin{array}{rcl}\RR^n&\to&\RR^p,\\~x_r&\mapsto& y \end{array}\] from the reduced state space to the output space directly, using RKHS methods. Given samples ${\bf s}=\{\Pi(x_j),y_j\}_{j=1}^{\ell}$, each coordinate function $\bigl(\hat{h}_i\bigr)_{i=1}^p$ is given in the familiar form  (\ref{learning1}) \[\hat{h}_i(\Pi(x)) = \sum_{j=1}^{\ell} b_j^iK^h\bigl(\Pi(x),\Pi(x_j)\bigr),\]
where $K^h$ is the kernel chosen to define the RKHS, and may be different for each coordinate and the $b_j$ satisfy the algebraic set of equations (\ref{learning2}). 

It should be noted that just given the state space reduction map $\Pi$, one can immediately compare the output of the system defined by $\hat{h}(x_r)$ to the original system without defining a closed dynamics as above. In fact with $\Pi$ and $\hat{h}$ one can design a simpler controller which takes as input the reduced state variable $x_r$, but controls the original system.

\subsection{On Error Estimates}\label{error_estimates}

The problems of approximating the functions $f$ and $h$ from time series  can be viewed as learning problem as described in  Appendix A and a theoretical justification of our algorithm is guaranteed by the error
estimates in Theorem \ref{thm:errors}. In fact, for the  dynamical
system $x(k+1)=f(x(k))$ (resp. the control system $x(k+1)=f(x(k),u(k))$), we have that $f^{\ast}$ in (\ref{unknown_func}) is the
map $f^{\ast}(x)=f_{i}(x)$ (resp. $f^{\ast}(z)=f_{i}(x,u)$ with $z=(x,u)'$)  and the samples
$\mathbf{s}$ in (\ref{samples}) are $(x(k),x_{i}(k+1)+\eta_{i})$ (resp.  $(\left(\begin{array}{c}x(k)\\u(k)\end{array}\right),x_{i}(k+1)+\eta_{i})$).

 Here $f_i$ is the unknown map $x(k) \rightarrow x_{i}(k+1) $ (resp. $\left(\begin{array}{c}x(k)\\u(k)\end{array}\right) \rightarrow x_{i}(k+1) $ )
 and it plays the role of the unknown function (\ref{unknown_func}).
  
  The initial
condition $x(0)$  (resp. the input) is known and $\eta_{i}$ are distributed according to a
probability measure $\rho_{x}$ that satisfies the following condition (this is
the \emph{Special Assumption} in \cite{smale_shannon1}).

\textbf{Assumption } The measure $\rho_{x}$ is the marginal on $X=\mathbb{R}%
^{n}$ of a Borel measure $\rho$ on $X\times\mathbb{R}$ with zero mean
supported on $[-M_{x},M_{x}],M_{x}>0$.

 Let's note that in \cite{smale_shannon1}, the authors do not consider time series and that we apply their results to time series. 
In the case of learning the map (\ref{output_map}) $\hat{h}\circ \Pi: \RR^n \rightarrow \RR^p$, we are looking at  learning the map $x_r \mapsto y$ from the samples ${\bf s}=\{\Pi(x_j),y_j\}_{j=1}^{\ell}$ while in the case of learning the vector-valued map $f:\RR^{n}\times\RR^m\to\RR^n$, this corresponds to learning the $i$-th coordinate of the map $(\tilde{x},u)\mapsto f(x,u)$ for $i=1,\cdots,n$.


In the case of the dynamics $x(k+1)=f^{\ast}(x(k))$ (resp. the control system $x(k+1)=f^{\ast}(x(k),u(k))$),   the error estimate (\ref{err_est1}) has the form
\begin{equation}
||\hat{x}_{i}(k+1)-x_{i}(k+1)||^{2}\leq2C_{\bar{x}}\mathcal{E}_{\mbox{samp}}%
+2||x(k+1)||_{K}^{2}(\gamma+8C_{\bar{x}}\Delta), \label{10}%
\end{equation}
where $||x_{i}(k+1)||_{\mathcal{H}_{K}}=\sum_{j=1}^{\infty}\frac{c_{i,j}^{2}%
}{\lambda_{j}}$.  

In the case of the unknown map (\ref{output_map}) $y = h\bigl(\Pi^{\dag}(x_r)\bigr)=\tilde{h}(x_r)$,   the error estimate (\ref{err_est1}) has the form
\begin{equation}
||\hat{y}-y||^{2}\leq2C_{\bar{x}}\mathcal{E}_{\mbox{samp}}
+2||\tilde{h}||_{K}^{2}(\gamma+8C_{\bar{x}}\Delta), \label{11}%
\end{equation}
where $||\tilde{h}||_{\mathcal{H}_{K}}=\sum_{j=1}^{\infty}\frac{\tilde{c}_{i,j}^{2}
}{\tilde{\lambda}_{j}}$.

The first term in the right hand side of inequalities (\ref{10})-(\ref{11})  represents the
error due to the noise (sampling error) and the second term represents the
error due to regularization (regularization error) and the finite-number of
samples (integration error).

\subsection{Algorithm Summary}
To summarize, the approach we have proposed proceeds as follows
 \begin{enumerate}
 \item Given a nonlinear control system (\ref{sigma}), let $u^i(t) = \delta(t)e_i$ be the $i$-th excitation signal for $i=1,\ldots,m$, and let $x^i(t):t\in[0,\infty)\mapsto x^i(t)\in\RR^n$ be the corresponding response of the system. Run the system and sample the trajectories at times $\{t_j\}_{j=1}^N$ to generate a collection of $N\cdot m$ vectors $\{x^i(t_j)\in\RR^n\}$.

\item Fixing $u(t) = 0$ and setting $x_0 = e_i$ for $i=1,\ldots,n$ (separately), measure the corresponding system output responses $y^i(t):t\in[0,\infty)\mapsto y^i(t)\in\RR^p$. As before, sample the responses at times $\{t_j\}_{j=1}^N$ and save the collection of $N\cdot p$ vectors $\{d_k(t_i)\}$ defined as
\[ d_k(t_j) = \bigl(y_k^1(t_j), \ldots, y_k^n(t_j)\bigr)^{\!\tr} \in\RR^n,\quad k=1,\ldots,p, \,\,j=1,\ldots,N\]

\item Choose a kernel $K$ defining a RKHS $\cH$, and form the Hankel kernel matrix 
$K_{o,c}\in\RR^{Np\times Nm}$,
\[(K_{o,c})_{\mu\nu} = K(d_\mu, x_\nu) \quad \mu=1,\ldots,Np, \,\, \nu=1,\ldots,Nm\]
where we have re-indexed the sets $\{d_k(t_i)\} = \{d_\mu\}, \{x^i(t_j)\} = \{x_\nu\}$
to use single indices.

\item Compute the eigendecomposition $K_{o,c}K_{o,c}^{\tr}=V\Sigma^2V^{\tr}$ assuming $K_{o,c}$ has been centered according to Equation~\eqref{eqn:koc_centering}.

\item The order of the model is reduced by discarding small eigenvalues $\{\Sigma_{ii}\}_{i=q+1}^n$, and projecting onto the subspace associated with the first $q<n$ largest eigenvalues. This leads to the state-space reduction map 
\[ \begin{array}{rcl}
\Pi:\RR^n&\rightarrow&\RR^q\\
    x  &\mapsto& x_r=\Pi(x) = T_q^{\tr}\mathbf{k}_o(x) 
\end{array}
\]
where $T_q = V_q\Sigma^{-1/2}_q$ and
$\mathbf{k}_o(x)$ is the centered empirical observability feature map given by Equation~\eqref{eqn:ko_centering}.

\begin{align} 
\widetilde{\mathbf{k}}_o(x) &= \bigl(\bPsi - \mu_o\bbone^{\tr}_M\bigr)^{\tr}\bigl(\Phi(x) - \mu_c\bigr) \nonumber\\
&= \mathbf{k}_o(x) - \tfrac{1}{N}K_{o,c}\bbone_N - \tfrac{1}{M}\bbone_M\bbone^{\tr}_M\mathbf{k}_o(x) + 
\tfrac{1}{NM}\bbone_M\bbone_M^{\tr}K_{o,c}\bbone_N.
\end{align}
and \[k_o(x)=\Psi^T \Phi(x)=(K(x,d_1),\cdots,K(x,d_M))^T \]

\item From input/output pairs or simulated/measured trajectories, learn approximations of the dynamics and output function defined on the reduced state space using, for instance, the representation theorem  in section \S \ref{sec:defins} (equations (\ref{learning1})-(\ref{learning2}). The RKHS used to approximate these functions need not be the same as the RKHS in which balanced truncation was carried out. 

For the approximation of the dynamics, the representer theorem will be applied to learn the
$i-th$ coordinate of the map $(\Pi(x),u) \mapsto f(x,u)$
using the samples ${\bf s}=(f_i(x_j,u_j),(\Pi(x_j),u_j))|_{j=1}^{\ell}$.

 For the approximation of the output map, the representer theorem will be applied to learn the map $x_r=\Pi(x) \mapsto y$ using the samples ${\bf s}=(\Pi(x_j),y_j)|_{j=1}^{\ell}$.

\item Approximate the Jacobian contribution as described in section \S \ref{sec:jacobian_app}. If the chosen kernels are polynomials, directly use (\ref{j_pi_x_poly}).
\item Combine the approximations to determine an expression for a closed, reduced, nonlinear dynamical system (\ref{reduced_system}) as described in sections \S. \ref{reduced_model} and  \S. \ref{output}.
\end{enumerate}


\section{Experiments}\label{sec:expts}

We demonstrate an application of our method on two examples appearing in~\cite{nilsson} (Examples 3.1. and 3.2, pgs. 52-54). 
\subsection{Two-Dimensional Exactly Reducible System}
Consider the nonlinear system
\[\begin{array}{rcl}
 \dot{x}_1&=&-3x_1^3+x_1^2x_2+2x_1x_2^2-x_2^3\\
 \dot{x}_2&=&2x_1^3-10x_1^2x_2+10x_1x_2^2-3x_2^3 - u\\
y&=&2x_1-x_2 \,.
\end{array}
\]
It can be shown that this system has the same input-output relationship as the system
$\dot{y}=-y^3+u$ by rearranging terms so that
\begin{eqnarray*}
 \dot{x}_1&=&-(2x_1-x_2)^2x_1+(x_1-x_2)^3\\
 \dot{x}_2&=& -(2x_1-x_2)^2x_2+2(x_1-x_2)^3-u\\
y&=&2x_1-x_2 \,.
\end{eqnarray*}
Defining the new variables $z_1=2x_1-x_2$ and $z_2=x_1-x_2$, the system can then be re-written
\[\begin{array}{rcl}
 \dot{z}_1&=&-z_1^3+u\\
 \dot{z}_2&=& -z_1^2z_2-z_2^3 + u\\
y&=&z_1 \,.
\end{array}
\]
It can be seen that the variable $z_2$ may be truncated because it doesn't appear in the expression of the output and thus doesn't affect $z_1$. Let's note here that the change of variables $x \mapsto z$ is a linear one and, importantly, that the method of nonlinear balancing as developed by Scherpen doesn't apply for this example since the system is not linearly controllable and the jacobian of the linearization has zero eigenvalues.

\subsection{Seven-Dimensional System}
We will also consider a 7-dimensional nonlinear system with one dimensional input and
output:
\[
\begin{array}{rcl}
\dot{x}_1 &=& -x_1^3 + u \\ 
\dot{x}_2 &=& -x_2^3 -x_1^2x_2 +3x_1x_2^2 -u \\
\dot{x}_3 &=& -x_3^3 + x_5 + u \\
 \dot{x}_4 &=& -x_4^3 + x_1 -x_2 +x_3 +2u \\
\dot{x}_5 &=& x_1x_2x_3 -x_5^3 + u  \\
 \dot{x}_6 &=& x_5 - x_6^3 -x_5^3 +2u\\
\dot{x}_7 &=& -2x_6^3 +2x_5 -x_7 -x_5^3 +4u \end{array}
\]
with $y = x_1 - x_2^2 + x_3 +x_4x_3 + x_5 -2x_6 +2x_7$.

Here also the method of nonlinear balancing as developed by Scherpen doesn't apply for this example since the system is not linearly controllable and the jacobian of the linearization has zero eigenvalues.

%
%
\subsection{Experimental Setup}
For both systems impulse and initial-condition responses of the system were simulated as described above, and 800 samples equally spaced in the time interval $[0,5s]$ were sampled to build the Hankel kernel matrix $K_{o,c}$ given by the third degree polynomial kernel $K(x,y)=(1 + \scal{x}{y})^3$. 
For the 2-D system we retained one component, and for the 7-D system we retained two for the sake of variety. Thus the reduction map $\Pi$ was defined by taking the top one or two eigenvectors (scaled columns of $T$) corresponding to the largest Hankel singular values, giving a reduced state space of dimension one or two for the 2-D and 7-D systems, respectively.

Next, a map from the reduced variable $x_r$ to $\dot{x}$ was estimated following Section~\ref{sec:f_rkhs}. The same procedure was followed in both experiments. The control input was chosen to be a 10hz square wave with peaks at $\pm 1$ at 50\% duty cycle, and 1000 samples from the simulated system in the interval $[0,5s]$ were mapped down using $\Pi$ and then used to solve the RLS regression problems, one for each state variable, again using a third degree polynomial kernel. All initial conditions were set to zero. The desired outputs (dependent variable examples) used to learn $\hat{f}$ were taken to be the true function $f$ evaluated at the samples from the simulated state trajectory. We also added a bias dimension of 1's to the data to account for an offset, and used a fast leave-one-out cross-validation (LOOCV) computation~\cite{RifRLS} to select the optimal regularization parameter. 

We followed a similar process to learn the output function $y = \hat{h}(x_r)$ for both systems. Here we used a 10Hz square wave control input (peaks at $\pm 2$, 50\% duty cycle), zero initial conditions and 700 samples in the interval $[0,5s]$. For this function the Gaussian kernel $K(x,y) = \exp(-\gamma\|x-y\|_2^2)$ was used to demonstrate that our method does not rely on any particular match between the form of the dynamics and the type of kernel. The scale hyperparameter $\gamma$ was chosen to be the reciprocal of the average squared-distance between the training examples. We again used LOOCV to select the RLS regularization parameter.

Finally, the comparisons between the exact and approximate reduced order systems were done using $x_0=0$ and a control input different from those used to learn the dynamics and output functions: $u(t) = \tfrac{1}{4}\bigl(\sin(2\pi 3t) + \text{sq}(2\pi 5t -\pi/2)\bigr)$ where $\text{sq}(\cdot)$ denotes the square wave function. 

The Taylor series approximation for $\Pi$ was done once, about $x_0$, and was not updated further. 

\subsection{Results}

For the 2D example, there is a clear gap in the singular values of $K_{oc}^TK_{oc}$, $\sigma_1(K_{oc}^TK_{oc})=89.3419$, $\sigma_2(K_{oc}^TK_{oc})=0.2574$ while the other ones are negligible.  For the 7D example, there is also a clear gap in the singular values of $K_{oc}^TK_{oc}$, $\sigma_1(K_{oc}^TK_{oc})=197.7821$, $\sigma_2(K_{oc}^TK_{oc})=46.9314$,  $\sigma_3(K_{oc}^TK_{oc})=2.7132$, $\sigma_4(K_{oc}^TK_{oc})=0.3293$, $\sigma_5(K_{oc}^TK_{oc})=0.0718$,  $\sigma_6(K_{oc}^TK_{oc})=0.0028$, $\sigma_7(K_{oc}^TK_{oc})=0.0010$, $\sigma_6(K_{oc}^TK_{oc})=0.0001$ while the other ones are negligible. We plot the first 100 singular values of the Hankel kernel matrix for both examples in  logarithmic scale in figure~\ref{fig:hankel_values_2d}.

The simulated outputs $\hat{y}(t)$ of the closed reduced systems as well as the output $y(t)$ of the original system together with a comparison with Lall et al. are plotted in Figures \ref{fig:sims2d} and \ref{fig:sims7d}, respectively. One can see that, even for a significantly different input, the reduced systems closely capture the original systems. The main source of error is seen to be over- and under-shoot near the square wave transients. This error can be further reduced by simulating the system for different sorts of inputs (and/or frequencies) and including the collected samples in the training sets used to learn $\Pi, \hat{f}$ and $\hat{h}$. Indeed, we have had some success driving example systems with random uniform input in some cases. Finally, we note that our method is as good as Lall et al.'s method especially that we assumed that the dynamics (\ref{eqn:nlsys}) is unknown and the simulation results are obtained using (\ref{reduced_system}) while in Lall et al. the dynamics is assumed to be known.

\begin{figure*}[t]
\centering
\includegraphics[height=8cm, width=15cm]{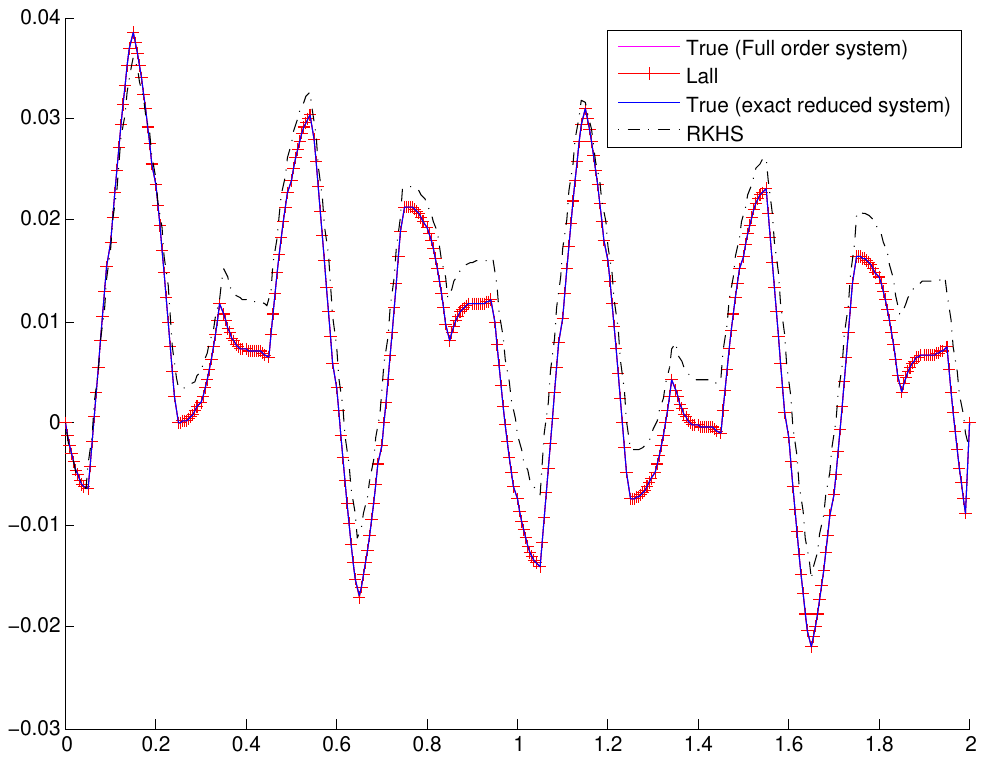}
\caption{{\em Simulated output trajectories for the original and reduced 2-dimensional system.}}
\label{fig:sims2d}
\end{figure*}

\begin{figure*}[t]
\centering
 
\includegraphics[height=8cm, width=15cm]{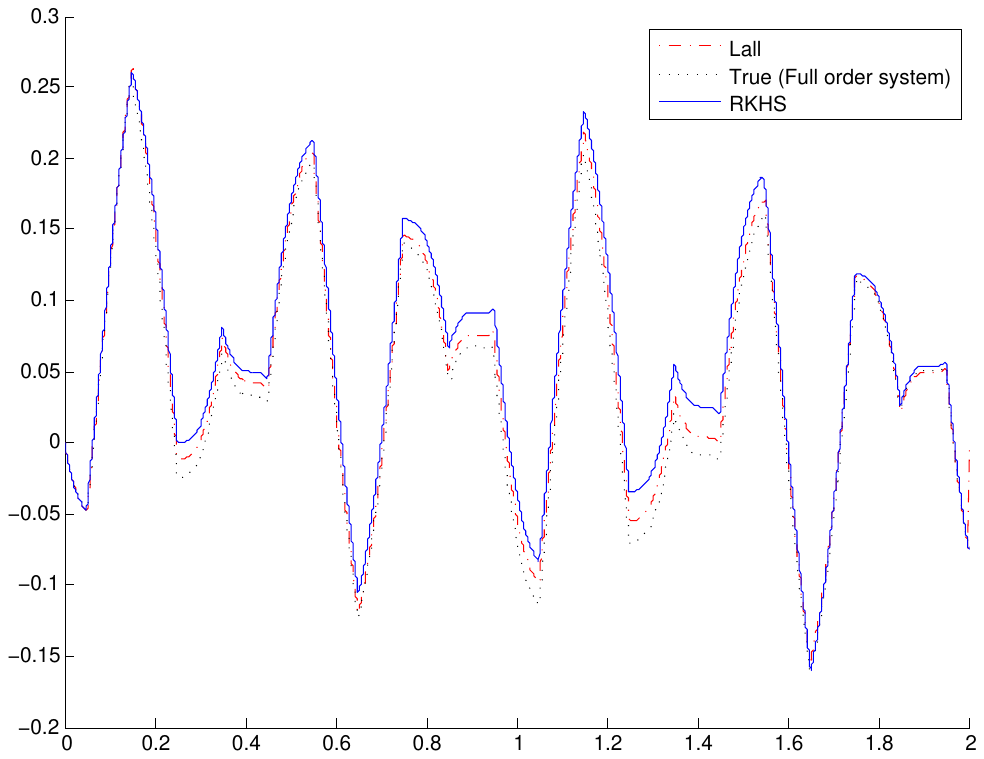}
\caption{{\em Simulated output trajectories for the original and reduced 7-dimensional system.}}
\label{fig:sims7d}
\end{figure*}


\begin{figure*}[t]
\centering
\includegraphics[height=8cm, width=15cm]{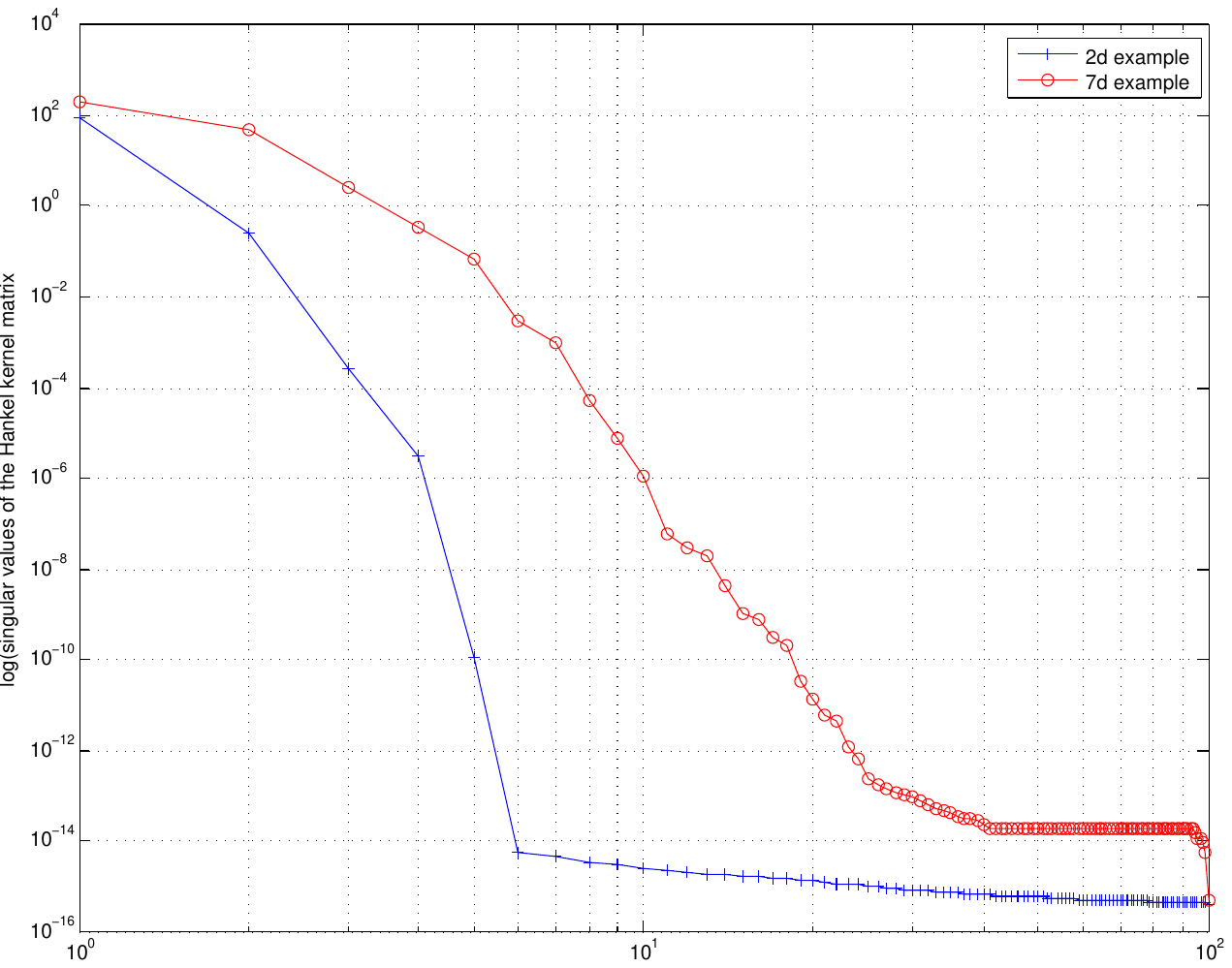}
\caption{{\em First hundred singular values of the Hankel kernel matrix, in logarithmic scale, for the 2D and 7D examples.}}
\label{fig:hankel_values_2d}
\end{figure*}


\section{Conclusion}
We have introduced a new, empirical model reduction method for nonlinear control systems. The method
assumes that the method of linear balancing applies to nonlinear systems in a high dimensional feature space. This leads to a nonlinear reduction map, which we suggest can be
combined with representations of the dynamics and output functions by elements of an RKHS to give a closed
reduced order dynamical system which captures the input-output characteristics of the original system. We then
demonstrated an application of our technique to a pair of nonlinear systems and simulated the original and
reduced models for comparison, showing that the approach proposed here can yield good low-order nonlinear
reductions of strongly nonlinear control systems. 
We believe that techniques well known to the machine
learning and statistics communities can offer much to control and dynamical systems research, and many further
directions remain, including computing error estimates, reduction of unstable systems, structure preserving
systems, stochastic differential equations (SDEs), and finding easily verifiable conditions of model reducibility of
nonlinear systems. For instance, we conjecture that a nonlinear system is model reducible if there exists an RKHS where there is a spectral gap in the gramians. Finally, our results allows us to argue that working in RKHSes  allows to develop methods  for a Data-Based Theory for (Nonlinear) Dynamical Systems.

\section{Acknowledgements}

BH thanks the European Commission and the Scientific and the Technological
Research Council of Turkey (Tubitak) for financial support received through a
Marie Curie Fellowship.
\appendix 

\section{Elements of Learning Theory} \label{sec:defins}



In this section, we give a brief overview of reproducing kernel Hilbert spaces as used in statistical learning
theory.
 The discussion here borrows heavily from~\cite{cucker,Wahba, smola_book}. Early work developing the theory of RKHS was
undertaken by I.J. Schoenberg \cite{schoenberg, schoenberg1, schoenberg2} and then N. Aronszajn~\cite{AronRKHS}. Historically, RKHSes came from the question: when is it possible to embed a metric space into a Hilbert space ?\footnote{A quasi-metric space $(X,d)$ is embeddable in a Hilbert space ${\cal H}$ if there exists a mapping $\Phi: X \rightarrow {\cal H}$ such that $d(x,y)=||\Phi(x)-\Phi(y)||$ for all $x$ and $y$ in X.  Schoenberg \cite{schoenberg, schoenberg1, schoenberg2} proved  that a finite metric space $(X,d)$ whose points are $x_0\cdots,x_n $ can be embedded into a Hilbert space  if and only if  the matrix $A$, whose entries are $A_{ij}=d(x_i,x_0)^2+d(x_j,x_0)^2-d(x_i,x_j)^2$, is nonnegative definite, i.e. $A$ is a Gram matrix \cite{cheney}.}

\begin{definition} Let  ${\cal H}$  be a Hilbert space of functions on a set ${\cal X}$. 
Denote by $\langle f, g \rangle$ the inner product on ${\cal H}$   and let $||f||= \langle f, f \rangle^{1/2}$
be the norm in ${\cal H}$, for $f$ and $g \in {\cal H}$. We say that ${\cal H}$ is a reproducing kernel
Hilbert space (RKHS) if there exists $K:{\cal X} \times {\cal X} \rightarrow \RR$
such that
\begin{itemize}
 \item[i.] $K$ has the reproducing property, i.e. $\forall f
\in {\cal H}$, $f(x)=\langle f(\cdot),K(\cdot,x) \rangle$.
\item[ii.] $K$ spans ${\cal H}$, i.e. ${\cal H}=\overline{\mbox{span}\{K(x,\cdot)|x \in {\cal X}\}}$.
\end{itemize}
$K$ will be called a reproducing kernel of ${\cal H}$. ${\cal H}_K(X)$  will denote the RKHS ${\cal H}$ 
with reproducing kernel  $K$.
\end{definition}

\begin{definition}
Given a kernel $K$  and inputs $x_1,\cdots,x_n \in {\cal X}$, the $n \times n$ matrix
\[k := (K(x_i,x_j))_{ij},\]
is called the \emph{Gram Matrix} of $k$ with respect to $x_1,\cdots,x_n$.  The kernel $K:{\cal X} \times {\cal X} \rightarrow \RR$ for which for all $n \in \NN$ and distinct $x_i \in {\cal X}$ gives rise to a strictly positive definite Gram matrix is called a strictly positive definite kernel.
\end{definition}

\begin{definition} (Mercer kernel map)
A function $K:{\cal X} \times {\cal X} \rightarrow \RR$ is called a Mercer kernel if it is continuous, 
symmetric and positive definite.
\end{definition}

The important properties of reproducing kernels are summarized in the following proposition
\begin{proposition}\label{prop1} If $K$ is a reproducing kernel of a Hilbert space ${\cal H}$, then
\begin{itemize}
\item[i.] $K(x,y)$ is unique.
\item[ii.]  $\forall x,y \in {\cal X}$, $K(x,y)=K(y,x)$ (symmetry).
\item[iii.] $\sum_{i,j=1}^m\alpha_i\alpha_jK(x_i,x_j) \ge 0$ for $\alpha_i \in \RR$ and $x_i \in {\cal X} $ (positive definitness).
\item[iv.] $\langle K(x,\cdot),K(y,\cdot) \rangle_{\cal H}=K(x,y)$.
\item[v.]
 Let $c \ne 0$. The following kernels, defined on a compact domain $\cX \subset \RR^n$, are Mercer kernels:
$K(x,y)=x\cdot y^{\prime}$ (Linear), $K(x,y)=(1+x\cdot y^{\prime})^d, \quad d \in \NN$ (Polynomial),
 $K(x,y)=e^{-\frac{||x-y||^2}{\sigma^2}}, \quad \sigma >0$ (Gaussian) .
\end{itemize}
\end{proposition}

\begin{theorem} \label{thm1}
Let $K:{\cal X} \times {\cal X} \rightarrow \RR$ be a symmetric and positive definite function. Then, there exists a Hilbert space of functions ${\cal H}$ defined on ${\cal X}$   admitting $K$ as a reproducing Kernel. Moreover, there exists a function $\Phi: X \rightarrow \mathcal{H}$ such that\footnote{This decomposition shows that kernels can be viewed as generalized dot products. }
$$K(x,y)= \langle \Phi(x), \Phi(y) \rangle_{\cal H} \quad \mbox{for} \quad x,y \in {\cal  X}.$$
 $\Phi$ is called a feature map\footnote{The dimension of the RKHS can be infinite and corresponds to the dimension of the eigenspace of the integral operator $L_K: {\cal L}_{\nu}^2({\cal X}) \rightarrow {\cal C}({\cal X})$ defined as $(L_K f)(x)=\int K(x,t)f(t)d\nu(t)$ if $K$ is a Mercer kernel, for $f \in {\cal L}_{\nu}^2({\cal X})$ and $\nu$ is a Borel measure on ${\cal X}$.}.

Conversely, let  ${\cal H}$ be a Hilbert space of functions $f: {\cal X} \rightarrow \RR$, with ${\cal X}$ compact, satisfying
$\forall x \in {\cal X}, \exists \kappa_x>0, \quad \mbox{such that} \quad |f(x)| \le \kappa_x ||f||_{\cal H}$.
Then, ${\cal H}$ has a reproducing kernel $K$.
\end{theorem}

\begin{theorem} \label{thm3}
 Every sequence of functions $(f_n)_{n \ge 1}$ which converges strongly to a function $f$ in ${\cal H}_K(X)$, converges also in the pointwise sense, that is, $\lim_{n \rightarrow \infty}f_n(x)=f(x)$, for any point $x \in X$. Further, this convergence is uniform on every subset of $X$ on which $x \mapsto K(x,x)$ is bounded.
\end{theorem}

\begin{remarks}
\begin{itemize}
\item[i.] In theorem \ref{thm1}, and using property [iv.] in Proposition \ref{prop1}, we can take $\Phi(x):=K_x:=K(x,\cdot)$ in which case $\mathcal{F}=\cH$ -- the ``feature space'' is the RKHS. This is called the \emph{canonical feature map}.
\item[ii.] The fact that Mercer kernels are positive definite and symmetric reminds us of similar properties of Gramians and covariance matrices. This is an essential fact that we are going to use in the following.
\item[iii.]  In practice, we choose a Mercer kernel, such as the ones in [v.] in Proposition \ref{prop1},  and theorem \ref{thm1} guarantees the existence of a Hilbert space admitting such a function as a reproducing kernel.

\item[iv.] Working in RKHSes allows to find nonlinear version of algorithms expressed in terms of inner products \cite{smola_book}. In fact, if an algorithm contains the quantity $\langle x, x^{\prime} \rangle$ then a nonlinear version of it, i.e. when $x$ is replaced by $\phi(x)$ with $\phi:\RR^n \rightarrow {\cal H}$, would contain the quantity $\langle \phi(x),\phi( x^{\prime}) \rangle$ where $\phi$ represents the nonlinearity. If we are working in an RKHS then $\langle \phi(x),\phi( x^{\prime}) \rangle := K(x,x^{\prime})$ and therefore we can replace all the quantities involving $\langle x, x^{\prime} \rangle$ in the original algorithm by $K(x,x^{\prime})$ in its nonlinear version.
\end{itemize}
\end{remarks}

\begin{example}
The following example is taken from \cite{berlinet}.  Let $V$ be the collection of functions $f$ with $f^{\prime \prime} \in L^2[0,1]$ and consider the subspace 
$$W_2^0=\{f(x) \in V: f, f^{\prime} \mbox{ absolutely continuous and } f(0)=f^{\prime}(0)=0 \}. $$
Define an inner product on $W_2^0$ as 
\[\label{eqn1} \langle f,g \rangle = \int_0^1 f^{\prime \prime}(t) g^{\prime \prime}(t) dt.  \]
Using integration by parts and the Fundamental Theorem of Calculus, it can be shown that for $f \in W_2^0$ and any $s \in [0,1]$, $f(s)$ can be written as
\[\label{eqn2} f(s)=\int_0^1 (s-u)_+ f^{\prime \prime}(u) du, \]
Since the reproducing kernel of the space $W_2^0$ must satisfy $f(s)=\langle f(\cdot), R(\cdot,s) \rangle$. From (\ref{eqn1}) and (\ref{eqn2}), we deduce that $K(\cdot,s)$ is a function such that 
$\frac{d^2 K(u,s)}{d^2 u}=(s-u)_+$.
Moreover, since $K(\cdot,s) \in W_2^0$ and using the property $K(s,t)=\langle K(\cdot,t), K(\cdot,s) \rangle$, we deduce that
\[K(s,t)=\langle K(\cdot,t), K(\cdot,s)\rangle=\int_0^1 (t-u)_+ (s-u)_+\, du=\frac{\mbox{max}(s,t)\, \mbox{min}^2(s,t)}{2}-\frac{\mbox{min}^3(s,t)}{6}  \]
\end{example} 

RKHS play an important role in learning theory whose objective is to find an
unknown function \[\label{unknown_func} f^{\ast}:X\rightarrow Y\] from random samples 

\[\label{samples1} {\bf s}=(x_{i},y_{i})|_{i=1}^{m},\]

In the following we review results from  \cite{smale_shannon1} (for a more general setting, cf. \cite{cucker}) about the special case when  the data samples ${\bf s}$ are such that 

\emph{Assumption 1:} The samples in (\ref{samples1}) have the special form
\[{\cal S: \quad}\label{samples}
{\bf s}=(x,y_x)|_{x \in \bar{x}},\]
where  $\bar{x}=\{x_i\}|_{i=1}^{d+1}$ and $y_x$ is drawn at random from $f^{\ast}(x)+\eta_x $, where $\eta_x$ is drawn from a probability measure $\rho_x$.

Here for each $x \in X$, $\rho_x$ is a probability measure with zero mean, and its variance $\sigma_x^2$ satisfies $\sigma^2 :=\sum_{x \in \bar{x}} \sigma_x^2 < \infty $. Let $X$ be a closed subset of $\RR^n$ and $\bar{t} \subset X$ is a discrete subset. Now, consider a kernel $K: X \times X \rightarrow \RR$ and define a matrix (possibly infinite) $K_{\bar{t},\bar{t}} : \ell^2(\bar{t}) \rightarrow \ell^2(\bar{t})$ as 
\[(K_{\bar{t},\bar{t}}a)_s = \sum_{t \in \bar{t}}K(s,t)a_t, \quad s \in \bar{t}, a \in \ell^2(\bar{t}), \]
where $\ell^2(\bar{t})$ is the set of sequences $a=(a_t)_{t \in \bar{t}}: \bar{t} \rightarrow \RR$ with $\langle a,b \rangle=\sum_{t \in \bar{t}}a_t b_t$ defining an inner product.  For example, we can take $X=\RR$ and $\bar{t}=\{0,1,\cdots,d\}$.

In the case of  dynamical systems such as the ones we are studying in this paper,  we are interested in learning the map ${\cal F}: x(k) \mapsto x(k+1)$ that characterises a dynamical system $x(k+1)={\cal F}(x(k))$ or  the map ${\cal F}: (x(k),u(k)) \mapsto x(k+1)$ that characterises a controlled dynamical system $x(k+1)={\cal F}(x(k),u(k))$. In order to get error estimates, we could easily apply the following results to the case of the dynamical systems we are interested in.

The case of approximating a function $f^{\ast} \in {\cal H}_K$ from samples of the form (\ref{samples1}) has been studied in \cite{smale_shannon1, smale_shannon2}. The problem we are interested in is the one of reconstructing $f^{\ast}$ from ${\bf s}$ which can be expressed as the minimisation problem
\[\label{reg_opt1}\bar{f}_{{\bf s},\gamma} :=\mbox{arg} {\mbox{min}}_{f \in {\cal H}_{K,\bar{t} } }\bigg\{\sum_{x \in \bar{x}}(f(x)-y_x)^2+\gamma ||f||_K^2\bigg\}, \]
where $\gamma \ge 0$. Moreover, in order to consider the case where $\bar{x}$ is not defined by a uniform grid on $X$, the authors of \cite{smale_shannon1} introduced a weighting $w := \{w_x\}_{x \in \bar{x}}$ on $\bar{x}$ with $w_x >0$\footnote{A suggestion proposed in \cite{smale_shannon1} is to consider the $\rho_X-$volume of the Voronoi associated with 
$\bar{x}$. Another example is $w=1$ or if $|\bar{x}|=m<\infty$, $w=\frac{1}{m}$.}.  Let $D_w$ be the diagonal matrix with ,aim diagonal entries $\{w_x\}_{x \in \bar{x}}$. Then, $||D_w|| \le ||w||_{\infty}$.

In this case, the regularisation scheme (\ref{reg_opt1}) becomes 
\[\label{reg_opt2}\bar{f}_{{\bf s},\gamma} :=\mbox{arg} {\mbox{min}}_{f \in {\cal H}_{K,\bar{t} } }\bigg\{\sum_{x \in \bar{x}}w_x(f(x)-y_x)^2+\gamma ||f||_K^2\bigg\}, \]

In learning theory, the minimization is taken over functions from a hypothesis space often taken to be a ball of a RKHS ${\cal H}_K$ associated to Mercer kernel $K$, and the function $f_{\bf s}$ that minimizes the empirical error ${\cal E}_{\bf s}$ is given by \emph {the representer theorem} (cf. \cite{Wahba, pontil, cucker} for derivation and more general forms of this theorem)

\begin{theorem} \label{thm_sol1}
Assume $f^{\ast} \in {\cal H}_{K, \bar{t}}$ and the standing hypotheses with $X$, $K$, $\bar{t}$, $\rho$ as above, $y$ as in (\ref{samples}). Suppose $K_{\bar{t},\bar{x}}D_wK_{\bar{x},\bar{t}} +\gamma K_{\bar{t},\bar{t}}$ is invertible. Define ${\cal L}$ to be the linear operator ${\cal L}=(K_{\bar{t},\bar{x}}D_wK_{\bar{x},\bar{t}}+\gamma K_{\bar{t},\bar{t}})^{-1}K_{\bar{t},\bar{x}}D_w$. Then the problem (\ref{reg_opt2}) has a unique solution 
\[\label{series1} f_{{\bf s},\gamma}=\sum_{t \in \bar{t}}({\cal L}y)_t K_t\]
\end{theorem}

The above theorem can be reformulated as follows

\begin{theorem}
Let ${\mathbf s} \in {\mathbb Z}^m$ and $\lambda \in \RR$, $\lambda >0$. The empirical target, i.e. the function $f_{\lambda,{\mathbf s}}=f_{\mathbf s}$ minimizing the regularized empirical error (\ref{reg_opt1})
over $f \in {\cal H}_K$, may be expressed as \[\label{learning1}f_{\bf s}(x)=\sum_{j=1}^mc_jK(x,x_j), \] where $c=(c_1,\cdots,c_m)$ is the unique solution of the well-posed linear system in $\RR^m$\[\label{learning2}\lambda \, m \, c_i+\sum_{j=1}^mK(x_i,x_j)c_j=y_i,\quad i=1,\cdots m, \]

\end{theorem}

\emph{Assumption 2}: For each $x \in X$, $\rho_x$ is a probability measure with zero mean supported on $[-M_x,M_x]$ with ${\cal B}_w :=(\sum_{x \in \bar{x}}w_x M_x^2)^{\frac{1}{2}} < \infty $.

\begin{definition} We say that $\bar{x}$ is $\Delta-$dense in $X$ if for each $y \in X$ there is some $x \in \bar{x}$ satisfying $||x-y||_{\ell^{\infty}(\RR^n)} \le \Delta$.
\end{definition}

\begin{theorem} (Sample Error) [Theorem 4, Propositions 2 and 3 in\cite{smale_shannon1}]

Suppose $K_{\bar{t},\bar{x}}D_wK_{\bar{x},\bar{t}} +\gamma K_{\bar{t},\bar{t}} $ is invertible.  Under the assumption (\ref{samples}), let $f_{{\bf s},\gamma}=\sum_{t \in \bar{t}}c_t K_t$ be the solution of (\ref{reg_opt2}) given in Theorem \ref{thm_sol1} by $c={\cal L}y$. Let ${\cal L}_w$ and $\kappa$ be
\[{\cal L}_w=(K_{\bar{t},\bar{x}}D_wK_{\bar{x},\bar{t}}+\gamma K_{\bar{t},\bar{t}})^{-1}K_{\bar{t},\bar{x}}D_w^{1/2} \]
 \[\kappa := ||K_{\bar{t},\bar{t}}|| \; ||(K_{\bar{t},\bar{x}}D_wK_{\bar{x},\bar{t}}+\gamma K_{\bar{t},\bar{t}})^{-1} ||^2\]

Then for every $0 < \delta <1$, with confidence $1-\delta$ we have the sample error estimate
\[\label{epsilon_samp} ||f_{{\bf s},\gamma}-f_{\bar{x},\gamma} ||_K^2 \le {\cal E}_{\mbox{samp}} := \kappa \sigma_w^2 \alpha^{-1} \bigg(\frac{2 ||K_{\bar{t},\bar{t}}{\cal L}_w|| \; ||{\cal L}_w|| \; {\cal B}_w^2 }{\kappa \sigma_w^2} \; \log{\frac{1}{\delta}}\bigg), \]
where $\alpha$ is the increasing function defined for $u >1$ as $\alpha(u)=(u-1) \log u$. In particular, $ {\cal E}_{\mbox{samp}} \rightarrow 0$ when $\gamma \rightarrow \infty$ or $\sigma^2_w \rightarrow 0$.

\end{theorem}

\begin{theorem} (Regularization Error and Integration Error, Proposition 4 and Theorem 5 in\cite{smale_shannon1})

Under Assumptions 1 and 2. Let $\bar{X}=(X_x)_{x \in \bar{x}}$ be the Voronoi of $X$ associated with $\bar{x}$ and $w_x=\rho_X(X_x)$.  Define the Lipschitz norm on a subset $X^{\prime} \subset X$ as $||f||_{\mbox{Lip}(X^{\prime})} := ||f||_{L^{\infty}(X^{\prime})}+\sup_{s,u \in X}\frac{|f(s)-f(u)|}{||s-u||_{\ell^{\infty}(\RR^n)}} $  and assume that the inclusion map of ${\cal H}_{K, \bar{t}}$ into the Lipschitz space satisfies\footnote{This assumption is true if $X$ is compact and the inclusion map of ${\cal H}_{K,\bar{t}}$ into the space of Lipschitz functions on $X$ is bounded which is the case when $K$ is a $C^2$ Mercer kernel \cite{zhou_capacity}. In fact, if $||f||_{\mbox{Lip}(X)} \le C_0 ||f||_K$ for each $f \in {\cal H}_{K,\bar{t}}$, then $C_{\bar{x}} \le C_0^2 \rho_X(X)$.}

\[C_{\bar{x}} := \sup_{f \in {\cal H}_{K, \bar{t}}}\frac{\sum_{x \in \bar{x}}w_x ||f||^2_{\mbox{Lip}(X_x)}}{||f||_K^2} < \infty. \]

\begin{itemize} 
\item[i.] If $f^{\ast} \in {\cal H}_{K,\bar{t}}$ and $\lambda_{\bar{x},w}>0$, then
\[||f_{\bar{x},\gamma}-f^{\ast} ||_K^2 \le \frac{\gamma ||K_{\bar{t},\bar{t}}|| \; ||f^{\ast}||_K^2}{\lambda_{\bar{x},w}^2} \] 
\item[ii.] If $\bar{x}$ is $\Delta-$dense, $C_{\bar{x}}< \infty$, and $f^{\ast} \in {\cal H}_{K,\bar{t}}$, then 
\[||f_{\bar{x},\gamma}-f^{\ast}||^2 \le ||f^{\ast}||_K^2 (\gamma+8 C_{\bar{x}}\Delta) \]

\end{itemize}
\end{theorem}

\begin{theorem} (Sample, Regularization and Integration Errors) \label{thm:errors} (Corollary 5 in \cite{smale_shannon1})

Under Assumptions 1 and 2. Let $\bar{X}=(X_x)_{x \in \bar{x}}$ be the Voronoi of $X$ associated with $\bar{x}$ and $w_x=\rho_X(X_x)$.  If $\bar{x}$ is $\Delta-$dense, $C_{\bar{x}}< \infty$, and $f^{\ast} \in {\cal H}_{K,\bar{t}}$, then, for every $0 < \delta < 1$, with probability at least  $1-\delta$ there holds 
\[\label{err_est1} ||f_{\bf{s},\gamma}-f^{\ast}||^2 \le 2C_{\bar x} {\cal E}_{\mbox{samp}}+2 ||f^{\ast}||_K^2 (\gamma+8 C_{\bar{x}}\Delta), \]
where ${\cal E}_{\mbox{samp}}$ is given in (\ref{epsilon_samp}).
\end{theorem}

 If $f^{\ast}$ is not an element of $ {\cal H}_{K,\bar{t}}$ then one also needs an estimate for the approximation error \cite{smale_approximation_error},\cite{cucker}. 

\begin{theorem}  Define $f_{\bf{s},\gamma}$ by (\ref{series1}). If $L_K^{-r}f_{\rho} \in L^2_{\rho_X}$, then 
\[||f_{\bf{s},\gamma}-f_{\rho}||_{L^2_{\rho_X}} \le  \lambda^r ||L_K^{-r}f_{\rho}||_{L^2_{\rho_X}},\quad \mbox{if} \quad 0 < r \le 1. \]
When $\frac{1}{2} < r \le 1$, we have 
\[||f_{\bf{s},\gamma}-f_{\rho}||_{K} \le \lambda^{r-\frac{1}{2}} ||L_K^{-r}f_{\rho}||_{L^2_{\rho_X}} \]

\end{theorem}

Here $f_{\bf s, \gamma}$ is taken as an approximation of the regression function $f_{\rho}$. Hence, minimizing over the (possibly infinite dimensional)
Hilbert space, reduces to minimizing over $\RR^m$.  The series (\ref{series1}) converges absolutely and uniformly to $f$.  We call \emph{learning} the process of approximating the unknown function $f$ from random samples on $Z$.



In the following, we assume that the kernels  $K$ are continuous and bounded by $\kappa = \sup_{x\in\cX}\sqrt{K(x,x)} < \infty $.

\section{Kernel PCA}
Kernel PCA~\cite{KPCA:98} will be a helpful starting point for understanding the approach to balanced reduction introduced in this paper. We briefly review the relevant background here.

Kernel PCA (KPCA) is a generalization of linear PCA that allows to take into account nonlinear versions of observations. This is done by carrying out PCA in a high dimensional  RKHS through
an injective, not necessarily surjective, map \[ \Phi: \RR^n \rightarrow {\cal H},  \; x \mapsto \Phi(x). \]  $\Phi$ is possibly nonlinear and ${\cal H}$ is a high-dimensional, possibly infinite-dimensional, RKHS.  Given the set of data $\mathbf{x}:=\{x_i\}_{i=1}^N \subset \RR^{n}$, the covariance matrix $C \in \RR^n \times \RR^n$ defined as 
\[\label{cov_mat} C=\frac{1}{N}\sum_{i=1}^N x_i x_i^T\] 
becomes a covariance matrix in ${\cal H}$
\[\label{eqn:kpca_cov} {\mathbf C_{\cal H}}=\frac{1}{N}\sum_{i=1}^N\Phi(x_i) \Phi(x_i)^T\]

If ${\cal H}$ is infinite-dimensional, we think of $\Phi(x_i) \Phi(x_i)^T$ as a linear operator on ${\cal H}$, mapping $x \mapsto \Phi(x_j) \langle \Phi(x_j),x\rangle$. We will assume the data are centered in the feature space so that $\sum_i \Phi(x_i) = 0$. If not, data may be centered according to the prescription in~\cite{KPCA:98}.

Taking the feature map\footnote{This feature map is valid given property iv. in Proposition A.3.} $\Phi: \RR^n \rightarrow \RR^N$ with  $\Phi_i(x)=K(x,x_i)$ and given the set of data $\mathbf{x}:=\{x_i\}_{i=1}^N \subset \RR^{n}$, we can consider PCA in $\RR^N$  by simply working with the covariance of the mapped vectors  (\ref{eqn:kpca_cov}).

The principal subspaces are computed by diagonalizing $C_{\cal H}$ through solving the eigenvalue problem
\[\label{ev_problem} \lambda v={\mathbf C}_{\cal H} v,\]
for eigenvalues $\lambda \ge 0$ and nonzero eigenvectors $v \in \RR^N$. This problem can be expressed in term of an inner-product by plugging (\ref{eqn:kpca_cov}) into (\ref{ev_problem}) and getting
\[\label{ev_problem1} \lambda v={\mathbf C}_{\cal H} v=\frac{1}{N}\sum_{i=1}^N\langle \Phi(x_i),v \rangle_{\cal H} \Phi(x_i),\]

However as is shown in~\cite{KPCA:98}, we can perform the computations directly in terms of the kernel without explicitly knowing $\Phi$.  One can equivalently form the matrix ${\mathbb K}$ of kernel
products whose entries are $({\mathbb K})_{ij} = K(x_i, x_j)$ for $i,j=1,\ldots,N$, and solve the eigenproblem 
\[\label{ev_problem_alpha} {\mathbb K}\boldsymbol{\alpha}=N\lambda\boldsymbol{\alpha},\]
 in $\RR^N$. Moreover the eigenvectors  $v$ in (\ref{ev_problem}) and the eigenvectors  $\boldsymbol{\alpha}$  in (\ref{ev_problem_alpha}) are related through  \[\label{v_and_alpha} v_i = \Psi\boldsymbol{\alpha}_i,\] where
$\Psi:=\bigl(\Phi(x_1)~\cdots~\Phi(x_M)\bigr)$, and the non-zero eigenvalues of ${\mathbb K}$ and $\mathbf C_{\cal H}$ coincide.

The eigenvectors $\boldsymbol{\alpha}_i$ of ${\mathbb K}$ are subsequently normalized so that the eigenvectors $v_i$ of $\mathbf C_{\cal H}$ have unit norm in the RKHS, leading to the condition $\nor{\boldsymbol{\alpha}_i}^2=\lambda_i^{-1}$. Assuming this normalization convention, sort the eigenvectors according to the magnitudes of the corresponding eigenvalues in descending order, and form the matrix \[A_q = \bigl[\boldsymbol{\alpha}_1 ~\cdots~ \boldsymbol{\alpha}_q\bigr], 1\leq q\leq \min(n,M).\]

Similarly, form the matrix $V_q = \bigl[v_1 ~\cdots~ v_q\bigr], 1\leq q\leq \min(n,M)$  of sorted eigenvectors of $C_{\cal H}$. The first $q$ principal components of a vector $x=\Phi(\tilde{x})$ in the feature space are then given by $V_q^{\tr}x$. It can be shown however (see~\cite{KPCA:98}) that principal components in the
feature space can be computed in the original space with kernels using the map $\Pi: \RR^M \rightarrow \RR^q$
 \[\label{map_kpca} \Pi(x) := A_q^{\tr}\mathbf{k}(x),\] where $\mathbf{k}(x) = \bigl(K(x,x_1),\ldots,K(x,x_M)\bigr)^{\tr}$.

Kernel methods can be used to develop nonlinear generalizations of any algorithm that can be expressed in terms of inner products \cite{KPCA:98} and KPCA is an illustration of this approach. KPCA is viewed as a nonlinear version of PCA since  PCA in $\RR^n$ can be reformulated as an eigenvalue problem in terms of inner products as in (\ref{ev_problem1}) but with $C_{\cal H}$ replaced by $C$ and $\Phi(x)=x$, i.e.
\[\label{ev_problem2}\lambda v = C v = \frac{1}{M} \sum_{i=1}^M \langle x_i, v \rangle x_i,\]
given the expression of $C$ in (\ref{cov_mat}).
If one wants to perform PCA on a nonlinear version of the data ${(x_i)}|_{i=1}^M$ through a nonlinear map $\Phi$, it is enough to replace $x$ by $\Phi(x)$ in the eigenvalue problem (\ref{ev_problem2}) to get (\ref{ev_problem1}).

Our  goal in this paper is to extend this method to linear balancing in view of applying it to nonlinear control systems.



\begin{thebibliography}{99}




\bibitem{antoulas} Antoulas, A. C. (2005). {\it Approximation of Large-Scale Dynamical Systems}, SIAM Publications.


\bibitem{archambeau} C Archambeau, D Cornford, M Opper, J Shawe-Taylor,  Gaussian process approximations of stochastic differential equations, Journal of machine learning research, pp. 1-16,	64,	2007.

\bibitem{AronRKHS} Aronszajn, N. (1950). {Theory of Reproducing Kernels}. {\it Trans. Amer. Math. Soc.}, 68:337-404.


\bibitem{berlinet} Berlinet, A.  and C. Thomas-Agnan (2004). {\it Reproducing Kernel Hilbert Spaces in Probability and
Statistics}, Kluwer Academic Publishers, Norwell, MA.

 \bibitem{allerton} Bouvrie, J.  and B. Hamzi (2010). Balanced Reduction of Nonlinear Control Systems in
Reproducing Kernel Hilbert Space, in {\it Proc. 48th Annual Allerton Conference on Communication, Control, and Computing}, pp. 294-301.
 \url{http://arxiv.org/abs/1011.2952}.
\bibitem{acc2012} Bouvrie, J.  and B. Hamzi (2012), Empirical Estimators for the Controllability Energy and
Invariant Measure of Stochastically Forced Nonlinear Systems,in {\it Proc. of the 2012 American Control Conference} (long version at
 \url{http://arxiv.org/abs/1204.0563}).

\bibitem{siam_jds} Bouvrie, J.  and B. Hamzi (2014), Embedology for Control and Random Dynamical Systems in Reproducing Kernel Hilbert Spaces, in preparation.

\bibitem{cheney} Cheney, W. and W. Light (2009). A Course in Approximation Theory, Graduate Studies in Mathematics, vol. 101, AMS.

\bibitem{Mauro08} Coifman, R. R., I. G. Kevrekidis, S. Lafon, M. Maggioni and B. Nadler (2008).
{Diffusion Maps, Reduction Coordinates, and Low Dimensional Representation of Stochastic Systems}, {\it
Multiscale Model. Simul.}, 7(2):842-864.

\bibitem{cucker} Cucker, F. and S. Smale (2001). On the mathematical foundations of learning. {\it Bulletin of AMS}, 39:1-49.

\bibitem{dullerud} Dullerud, G. E., and F. Paganini (2000). {\it A Course in Robust Control Theory: a Convex Approach}, Springer.

\bibitem{pontil} Evgeniou, T., M. Pontil, and T. Poggio (2000). 
Regularization networks and support vector machines,
Advances in Computational Mathematics, vol. 13, no. 1, pp. 1-50.


\bibitem{fujimoto} Fujimoto, K. and D. Tsubakino (2008). { Computation of nonlinear balanced realization and model reduction based on Taylor series expansion}, {\it Systems and Control Letters},
{\bf 57},  4, pp. 283-289.

\bibitem{fujimoto1} Fujimoto, K. and J. Scherpen (2010).  Balanced realization and model order reduction for nonlinear systems based on singular value analysis, SIAM Journal on Control and Optimization, Vol. 48, 7, pp. 180-194.

\bibitem{gray} Gray, W. S. and  E. I. Verriest (2006).  Algebraically Defined Gramians for Nonlinear Systems, {\it Proc. of the 45th IEEE CDC}.

\bibitem{hammarling} Hammarling, S. J.  (1982). Numerical Solution of the Stable, Non-negative Definite
Lyapunov Equation,  IMA Journal of Numerical Analysis, vol. 2, pp. 303-323.

 \bibitem{jolliffe} Jolliffe, I.T. (2002). Principal Component Analysis, Springer.

\bibitem{jonckheere} Jonckheere, E.A., and L. M. Silverman (1983). {\it A New Set of Invariants for Linear Systems - Application to Reduced Order Compensator Design}. {IEEE Transactions on Automatic Control}, {\bf AC-28}, 10,  953-964.

\bibitem{kantz}   Kantz, H. and T. Schreiber. Nonlinear Time Series Analysis, Cambridge University Press, 2004.


\bibitem{kenney} Kenney, C. and G. Hewer (1987). {\it Necessary and Sufficient Conditions for Balancing Unstable Systems}, IEEE Transactions on Automatic Control, {\bf 32}, 2, pp. 157-160.

\bibitem{krener} Krener, A. (2006). Model reduction for linear and nonlinear control systems. Bode Lecture, 45th IEEE Conference on Decision and Control.

\bibitem{krener2} Krener, A. J. (2007). The Important State Coordinates of a Nonlinear System. In {\it ``Advances in control theory and applications''}, C. Bonivento, A. Isidori, L. Marconi, C. Rossi, editors, pp. 161-170. Springer.
\bibitem{krener1} Krener, A. J. (2008). Reduced order modeling of nonlinear control systems. In  {\it ``Analysis and Design of Nonlinear Control Systems''}, A. Astolfi and
L. Marconi, editors, pp. 41-62. Springer.


\bibitem{Kwok:ICML:03} Kwok, J. T. and I.W.~Tsang (2003). ``The Pre-Image Problem in Kernel Methods''.
In {\it Proceedings of the Twentieth International Conference on Machine Learning (ICML)}.

 \bibitem{lall} Lall, S., J. Marsden, and S. Glavaski (2002). { A subspace approach to balanced truncation for model reduction of
 nonlinear control systems}, {\it International Journal on Robust and Nonlinear Control}, {\bf 12}, 5, pp. 519-535.

 \bibitem{laub} Laub, A.J. (1980). On Computing ``balancing'' transformations, {\it Proc. of the 1980 Joint Automatic Control Conference (ACC)}.

\bibitem{li} Li, J.-R. (2000). {\it Model Reduction of Large Linear Systems via Low Rank System Grammians}. Ph.D. thesis, Massachusetts Institute of Technology.

\bibitem{mika98} Mika, S., B. Sch\"{o}lkopf, A. Smola, K. R. M\"{u}ller, M. Scholz, and G. R\"{a}tsch (1998).  Kernel PCA and de-noising in feature spaces, In {\it Proc. Advances in Neural Information Processing Systems (NIPS) 11}, pp. 536--542, MIT Press.

\bibitem{moore} Moore, B. (1981). { Principal Component Analysis in Linear Systems: Controllability, Observability, and Model Reduction},
{\it  IEEE Tran. Automat. Control}, {\bf 26}, 1, pp. 17-32.

 \bibitem{newman} Newman, A.J., and P. S. Krishnaprasad (2000). Computing balanced realizations for nonlinear systems, {\it Proc. of the Math. Theory of Networks
 and Systems (MTNS)}.
\bibitem{nilsson} Nilsson, O. (2009). {\it On Modeling and Nonlinear Model Reduction in Automotive Systems}, Ph.D. thesis, Lund University.

\bibitem{Phillips} Phillips, J.,  J.~Afonso, A.~Oliveira and L. M. Silveira (2003). {Analog Macromodeling using Kernel Methods}. In {\it Proceedings of the IEEE/ACM International Conference on Computer-aided Design}.


\bibitem{RifRLS} Rifkin, R., and R.A. Lippert. {\it Notes on Regularized Least-Squares}, CBCL Paper 268/AI Technical Report 2007-019, Massachusetts Institute of Technology, Cambridge, MA, May, 2007.



\bibitem{schaback_survey} Schaback, R.  and H. Wendland. Kernel techniques: From machine learning to meshless methods, Acta Numerica, vol.15, pp. 543-639, 2006.

\bibitem{scherpen_balancing} Scherpen, J. (1993). Balancing for nonlinear systems. Syst. \& Contr. Let., 21, pp 143-153.

\bibitem{scherpen_thesis} Scherpen, J. (1994). {\it Balancing for Nonlinear Systems}, Ph.D. thesis, University of Twente, \url{http://www.dcsc.tudelft.nl/~jscherpen/thesis.html}.

\bibitem{scherpen_survey} Scherpen, J. (2011). Balanced Realizations, Model Order Reduction, and the Hankel Operator,  the Control Systems Handbook, 2nd ed., Advanced Methods, Eds. William Levine, CRC Press, Taylor \& Francis Group, Chap. 4, pp. 1-24.

\bibitem{schoenberg} Schoenberg, I. J. (1935). Remarks to Maurice Fr\'echet's article "Sur la d\'efinition axiomatique d'une classe d'espace distanci\'es vectoriellement applicable sur l'espace de Hilbert", Annals of Mathematics 36 (1935), 724--732.

\bibitem{schoenberg1} Schoenberg, I. J.  (1937) On certain metric spaces arising from euclidean spaces by a change of metric
and their imbedding in Hilbert space, Annals of Mathematics, (2), vol. 38 (1937), pp. 787-793.
\bibitem{schoenberg2} Schoenberg, I. J. (1938). Metric spaces and positive definite functions,  Trans. Amer. Math. Soc. 44 (1938), pp. 522-536.

\bibitem{KPCA:98}
Sch\"{o}lkopf, B., Smola, A., and M\"{u}ller, K (1998).
\newblock Nonlinear component analysis as a kernel eigenvalue problem.
\newblock \emph{Neural Computation}, 10\penalty0 (5):\penalty0 1299--1319.

\bibitem {smale_approximation_error} Smale, S. and D.,-X. Zhou (2003).
Estimating the Approximation Error in Learning Theory, Analysis and
Applications, Volume 01, Issue 01.

\bibitem {smale_shannon1}  Smale, S. and D.,-X. Zhou, Shannon sampling and function reconstruction from point values, Bull. Amer. Math. Soc. 41, 2004, pp. 279-305.

\bibitem {smale_shannon2}  Smale, S. and D.,-X. Zhou,  Shannon sampling II: Connections to learning theory, Applied and Computational Harmonic Analysis, Volume 19, Issue 3, November 2005, pp. 285-302.

\bibitem{hyperbolic} S.~Smale and D.-X.~Zhou, \textsl{Online Learning with Markov Sampling}, Anal. Appl., 7, pp. 87-113,  2009.


\bibitem{smola_book} Sch\"{o}lkopf, B.,  and A. J. Smola  (2002). Learning with Kernels, The MIT Press.





\bibitem{steinwart_svms}   Steinwart, I. and A. Christmann (2008). Support Vector Machines, Springer.

\bibitem{therapos} Therapos, C. P. (1989). {\it Balancing Transformations for Unstable Nonminimal Linear Systems}, IEEE Transactions on Automatic Control, {\bf 34}, 4, pp. 455-457.



\bibitem{verriest1} Verriest, E. (1981). Suboptimal LQG-Design via Balanced Realizations. Proc. of the 20th IEEE CDC, pp. 686-687.

\bibitem{verriest2} Verriest, E.  (1984). Approximation and Order Reduction in Nonlinear models using an RKHS Approach, Proc. of the 18th Annual Conference on Information Sciences and Systems, pp. 197-201.

\bibitem{box}  Box, G. E. P.  and G. M. Jenkins, Time Series Analysis: Forecasting and Control,  Wiley, 2008.

 \bibitem{ljung} L. Ljung, System Identification: Theory for the User, Prentice Hall,   1999. 
  

\bibitem{Wahba} Wahba, G. (1990). {Spline Models for Observational Data}, {\it SIAM CBMS-NSF Regional Conference Series in Applied Mathematics} 59.

\bibitem{weiland} Weiland, S. (1991). Theory of approximation and disturbance attenuation of linear systems, Doctoral dissertation, University of Groningen.

\bibitem{zhou_capacity} Zhou, D.-X. (2003). 
Capacity of reproducing kernel spaces in learning theory, IEEE Transactions on Information Theory,   Volume:49 ,  Issue: 7, pp. 1743 - 1752.















\end{thebibliography}
\end{document}